\documentclass{article}

\usepackage{amsmath,amssymb,amsthm}
\usepackage{tikz}
\usepackage{color}
\usepackage[toc]{appendix}
\usepackage{graphicx}
\usepackage{fancyhdr}
\usepackage{enumitem}
\usepackage{bbm}
\usepackage{parskip}
\usepackage{float}
\usepackage{chngpage}
\usepackage{calc}
\usepackage{bigints}
\usepackage{array}
\usepackage{booktabs}
\usepackage{rotating}
\usepackage{multirow}
\usepackage{adjustbox}
\usepackage{tabularx}
\usepackage{verbatim}
\usepackage{mathtools}
\usepackage{ragged2e}
\usepackage[makeroom]{cancel}
\usepackage{caption}
\usepackage{hyperref}
\usepackage{caption}
\usepackage{subcaption}
\usepackage{appendix}
\usepackage{pgfplots}
\pgfplotsset{compat=1.16}
\usepackage[english]{babel}
\usepackage{hyphenat}
\usepackage[makeindex]{imakeidx}
\usetikzlibrary{datavisualization}
\usetikzlibrary{matrix}
\usetikzlibrary{datavisualization.formats.functions}

\newtheorem{theorem}{Theorem}[section]
\newtheorem{proposition}[theorem]{Proposition}

\newtheorem{remark}[theorem]{Remark}
\newtheorem{assumption}[theorem]{Assumption}

\makeatletter
\def\section{\@startsection {section}{1}{\z@}{3.25ex plus 1ex minus
		.2ex}{1.5ex plus .2ex}{\large\bf}}
\def\subsection{\@startsection{subsection}{2}{\z@}{3.25ex plus 1ex minus
		.2ex}{1.5ex plus .2ex}{\normalsize\bf}}
\@addtoreset{equation}{section} 
\makeatother

\title{Using Malliavin calculus to solve a chemical diffusion master equation}

\author{Alberto Lanconelli\thanks{Dipartimento di Scienze Statistiche Paolo Fortunati, Università di Bologna, Bologna, Italy. \textbf{e-mail}: alberto.lanconelli2@unibo.it}}

\date{\today}

\begin{document}
	
\maketitle
	
\bigskip
	
\begin{abstract}
We propose a novel method to solve a chemical diffusion master equation of birth and death type. This is an infinite system of Fokker-Planck equations where the different components are coupled by reaction dynamics similar in form to a chemical master equation. This system was proposed in \cite{del Razo} for modelling the probabilistic evolution of chemical reaction kinetics associated with spatial diffusion of individual particles. Using some basic tools and ideas from infinite dimensional Gaussian analysis we are able to reformulate the aforementioned infinite system of Fokker-Planck equations as a single evolution equation solved by a generalized stochastic process and written in terms of Malliavin derivatives and differential second quantization operators. Via this alternative representation we link certain finite dimensional projections of the solution of the original problem to the solution of a single partial differential equations of Ornstein-Uhlenbeck type containing as many variables as the dimension of the aforementioned projection space.    
\end{abstract}
	
	
	
\allowdisplaybreaks

\section{Introduction and statement of the main result}\label{intro}

Suppose we want to model the probabilistic evolution of a system that is initially constituted by a single particle of a chemical species $A$, which is located somewhere in the interval $[0,1]$ according to a given probability density function $\zeta:[0,1]\to\mathbb{R}$, and that undergoes: 
\begin{itemize}
	\item \emph{degradation} and \emph{creation} chemical reactions
\begin{align}\label{reaction}
	\mbox{(I)}\quad A\xrightarrow{\lambda_d(x)}\varnothing\quad\quad\mbox{(II)}\quad \varnothing\xrightarrow{\lambda_c(x)}A,
\end{align}
where $\lambda_d(x)$ denotes the propensity for reaction (I) to occur for a particle located at position $x\in [0,1]$ (i.e., the probability per unit of time for this particle to disappear) while $\lambda_c(x)$ is the propensity for a new particle to be created at position $x\in [0,1]$ by reaction (II);
\item drift-less isotropic \emph{diffusion} in space.
\end{itemize}
While reactions \eqref{reaction} alone can be analysed via the standard chemical master equation \cite{Gillespie},\cite{Lecca},\cite{Van Kampen} and the sole diffusive motion of the particles through a Fokker-Planck equation \cite{Allen},\cite{Gardiner}, the combination of these two phenomena makes the mathematical description quite challenging. This is due to the hybrid nature of the considered reaction-diffusion process, i.e. discrete in the evolution of the number of  particles (and hence of the spatial dimension of the problem) and continuous in the random movement of those particles. In the recent paper \cite{del Razo}, the authors proposed a set of equations for the functions
\begin{align*}
	\rho_n(t,x_1,...,x_n):=p_n(t,x_1,...,x_n)\mathbb{P}(N(t)=n),\quad n\geq 0,
\end{align*}
that aims to model the reaction-diffusion process described above. Here $p_n(t,x_1,...,x_n)$ represents the joint probability density function, conditioned to the event $\{N(t)=n\}$, i.e. the number of particles  at time $t$ is equal to $n$, for the positions of these particles at time $t$. The model is named \emph{chemical diffusion master equation} of birth and death type and it takes the form:
\begin{align}\label{equation}
\begin{split}
\partial_t\rho_0(t)=&\int_0^1\lambda_d(y)\rho_1(t,y)dy-\int_0^1\lambda_c(y)dy \cdot\rho_0(t),\quad t>0;\\
\partial_t\rho_n(t,x_1,...,x_n)=&\sum_{i=1}^n\partial^2_{x_i}\rho_n(t,x_1,...,x_n)\\
&+(n+1)\int_0^1\lambda_d(y)\rho_{n+1}(t,x_1,...,x_n,y)dy\\
&-\sum_{i=1}^n\lambda_d(x_i)\rho_n(t,x_1,...,x_n)\\
&+\frac{1}{n}\sum_{i=1}^n\lambda_c(x_i)\rho_{n-1}(t,x_1,...,x_{i-1},x_{i+1},...,x_n)\\
&-\int_0^1\lambda_c(y)dy\cdot\rho_n(t,x_1,...,x_n),\quad n\geq 1, t>0, (x_1,...,x_n)\in ]0,1[^n,
\end{split}
\end{align}
with initial and Neumann boundary conditions 
\begin{align}\label{initial}
\begin{split}
\rho_0(0)&=0;\\
\rho_1(0,x_1)&=\zeta(x_1),\quad x_1\in[0,1];\\
\rho_n(0,x_1,...,x_n)&=0,\quad n> 1, (x_1,...,x_n)\in[0,1]^n;\\
\partial_{\nu}\rho_n(t,x_1,...,x_n)&=0,\quad n\geq 1, t\geq 0, (x_1,...,x_n)\in\partial [0,1]^n.
\end{split}
\end{align}
This is an infinite system of Fokker-Planck equations where the components have an increasing number of degrees of freedom (to account for all the possible numbers of particles in the system) and are coupled through the reaction mechanism \eqref{reaction}. The term
\begin{align*}
	\sum_{i=1}^n\partial^2_{x_i}\rho_n(t,x_1,...,x_n)
\end{align*}
in \eqref{equation} refers to the drift-less isotropic spatial diffusion; the terms
\begin{align*}
	(n+1)\int_0^1\lambda_d(y)\rho_{n+1}(t,x_1,...,x_n,y)dy-\sum_{i=1}^n\lambda_d(x_i)\rho_n(t,x_1,...,x_n) 
\end{align*}
formalize gain and loss, respectively, due to reaction (I), while
\begin{align*}
	\frac{1}{n}\sum_{i=1}^n\lambda_c(x_i)\rho_{n-1}(t,x_1,...,x_{i-1},x_{i+1},...,x_n)-\int_0^1\lambda_c(y)dy\cdot\rho_n(t,x_1,...,x_n)
\end{align*}
relate to reaction (II). The functions $\lambda_d,\lambda_c,\zeta:[0,1]\to\mathbb{R}$ in \eqref{equation}-\eqref{initial} are assumed to be non negative, bounded and measurable; in addition, $\zeta$ is smooth with $\zeta'(0)=\zeta'(1)=0$ and $\int_{[0,1]}\zeta(x)dx=1$. The symbol $\partial_{\nu}$ in \eqref{initial} stands for the directional derivative along the outer normal vector at the boundary of $[0,1]^n$. The particles are assumed to be indistinguishable thus entailing the symmetry of $p_n(t,x_1,...,x_n)$, and hence of $\rho_n(t,x_1,...,x_n)$, in all their space variables. We observe that by construction,
\begin{align*}
\int_{[0,1]^n}p_n(t,x_1,...,x_n)dx_1\cdot\cdot\cdot dx_n=1, 
\end{align*}
and hence
\begin{align}\label{s}
\int_{[0,1]^n}\rho_n(t,x_1,...,x_n)dx_1\cdot\cdot\cdot dx_n=\mathbb{P}(N(t)=n).
\end{align}
Therefore, since $\sum_{n\geq 0}\mathbb{P}(N(t)=n)=1$, the solution to \ref{equation} should fulfil
\begin{align}\label{mass preservation}
\sum_{n\geq 0}\int_{[0,1]^n}\rho_n(t,x_1,...,x_n)dx_1\cdot\cdot\cdot dx_n=1.
\end{align}
But this is indeed the case; in fact, if we assume the functions $\lambda_d$ and $\lambda_c$ to be constant and we integrate out the spatial degree of freedom in \eqref{equation}, we see that the diffusive part vanishes by virtue of Gauss theorem combined with the Neumann boundary condition in \eqref{initial}; moreover, recalling \eqref{s} we get
\begin{align}\label{master}
	\begin{split}
\partial_t\mathbb{P}(N(t)=n)=&(n+1)\lambda_d\mathbb{P}(N(t)=n+1)-n\lambda_d\mathbb{P}(N(t)=n)\\
&+\lambda_c\mathbb{P}(N(t)=n-1)-\lambda_c\mathbb{P}(N(t)=n).
\end{split}
\end{align}
Equation \eqref{master} is exactly the chemical master equation as derived from the classical law of mass action for
reactions \eqref{reaction} in spatially well-mixed systems \cite{Gillespie}; this fact proves condition \eqref{mass preservation} and the consistency of model \eqref{equation}-\eqref{initial} with classical equations of chemical kinetics.

Aim of the present paper is to propose a novel method to analytically solve system \eqref{equation}-\eqref{initial}. Our idea consists in transforming the deterministic problem \eqref{equation}-\eqref{initial} into an equivalent stochastic one by using iterated It\^o integrals with respect to a one dimensional Brownian motion and the Wiener-It\^o chaos expansion theorem. Then, exploiting the Gaussian framework induced by the aforementioned It\^o integration, we obtain a single partial differential equation of Ornstein-Uhlenbeck type in infinitely many variables which is directly connected to the problem \eqref{equation}-\eqref{initial} via the so-called Stroock-Taylor formula: this provides an analytical representation for the solution of the original system. \\
The main steps of our approach are summarized as follows:
\begin{itemize}
	\item we assume the existence of a classical solution $\{\rho_n\}_{n\geq 0}$ for \eqref{equation}-\eqref{initial}; the continuity of $\rho_n(t,\cdot)$ together with its symmetry in the spatial variables imply the membership of $\rho_n(t,\cdot)$ to $L_s^2([0,1]^n)$, the space of symmetric square integrable functions;
	\item we It\^o-integrate all the space variables of $\rho_n(t,\cdot)$ with respect to a one dimensional Brownian motion $\{B_x\}_{x\in [0,1]}$; this produces a sequence of multiple It\^o integrals $\{I_n(\rho_n(t,\cdot))\}_{n\geq 0}$ which solves a new set of equations, equivalent to \eqref{equation} but expressed in terms of differential second quantization operators, Malliavin derivatives and their adjoints (see equation \eqref{equation 5} below);
	\item we define $\Phi(t):=\sum_{n\geq 0}I_n(\rho_n(t;\cdot))$; this is a generalized stochastic process, solution to a single infinite dimensional differential equation, see \eqref{SDE infinite} below, and whose kernels (from its Wiener-It\^o chaos expansion) are by construction the elements of the sequence $\{\rho_n\}_{n\geq 1}$. In this Gaussian framework one can rewrite the adjoint Malliavin derivative, which appears in the equation for $\{\Phi(t)\}_{t\geq 0}$, as a difference between multiplication operator and Malliavin derivative;
	\item using this transformation, the equation for $\{\Phi(t)\}_{t\geq 0}$ is reduced to a single parabolic equation of Ornstein-Uhlenbeck-type with infinitely many variables; the kernels of the solution to suitable finite dimensional projections of such equation provide, through the Stroock -Taylor formula, an analytic representation for the corresponding projections of the sequence $\{\rho_n\}_{n\geq 0}$ (see formula \eqref{formula} below).
\end{itemize}  

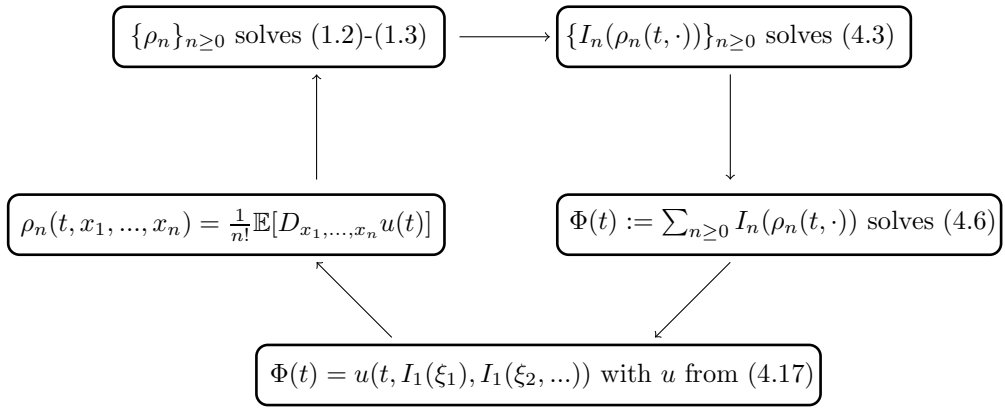
\begin{figure}
	
	\begin{center}
		
		\begin{tikzpicture}
			
			\draw[rounded corners=1ex,line width=1pt] (-2.2,5.4) rectangle (2.2,4.6);
			\draw[rounded corners=1ex,line width=1pt] (3.7,5.4) rectangle (8.4,4.6);
			\draw[rounded corners=1ex,line width=1pt] (3.7,2.9) rectangle (9.6,2.1);
			\draw[rounded corners=1ex,line width=1pt] (-3.6,2.9) rectangle (2.2,2.1);
			\draw[rounded corners=1ex,line width=1pt] (-0.3,0.1) rectangle (7.2,0.9);
	
			\node at (0, 5) {$\{\rho_n\}_{n\geq 0}$ solves \eqref{equation}-\eqref{initial}};
			\node at (6, 5) {$\{I_n(\rho_n(t,\cdot))\}_{n\geq 0}$ solves \eqref{equation 5}};
			\node at (6.7, 2.5) {$\Phi(t):=\sum_{n\geq 0}I_n(\rho_n(t,\cdot))$ solves \eqref{SDE infinite}};
			\node at (-0.7, 2.5) {$\rho_n(t,x_1,...,x_n)=\frac{1}{n!}\mathbb{E}[D_{x_1,...,x_n}u(t)]$};
			\node at (3.5, 0.5) {$\Phi(t)=u(t,I_1(\xi_1),I_1(\xi_2,...))$ with $u$ from \eqref{last}};

			\draw [->,out=0,in=180] (2.4, 5) to (3.6,5);
			\draw [->,out=270,in=90] (6, 4.5) to (6,3.1);
			\draw [->,out=225,in=45] (6, 2) to (5,1);
			\draw [->,out=135,in=315] (1.5,1) to (0.5,2);
			\draw [->,out=90,in=270] (0.5, 3.1) to (0.5, 4.5);

		\end{tikzpicture}
		
	\end{center}
	
	\caption{Representation of $\{\rho_n\}_{n\geq 0}$ as Wiener-It\^o kernels of the stochastic process $\{\Phi(t)\}_{t\geq 0}$}
	\label{fig: diagram}
\end{figure}

Problem \eqref{equation}-\eqref{initial} corresponds to one particular instance of the class of problems formalized in the recent paper \cite{del Razo}. The aim of that work was to develop a general framework for stochastic particle-based reaction-diffusion processes in the form of an evolution equation for the probability density of the open system. We refer readers to the nice account, presented in \cite{del Razo}, of the vaste literature concerning the mathematical modelling of chemical and biochemical phenomena which combine diffusion and chemical reactions.\\
Compared to the model from \cite{del Razo} which system \eqref{equation}-\eqref{initial} refers to, we made some simplifying assumptions. In \cite{del Razo} the authors assume that the single particle can move in the open bounded region $\mathbb{X}\subset\mathbb{R}^3$, so that $\rho_n$ is defined on $[0,+\infty[\times\mathbb{X}^{\otimes n}$; in addition, they deal with a non necessarily isotropic diffusion with drift term. Here, for simplicity we focus on a one dimensional case, i.e. $\mathbb{X}=]0,1[$, and assume drift-less isotropic diffusion but our technique readily extends to the general case. \\
Contrary to the $L^1$-framework adopted in \cite{del Razo} and utilized to set up the Fock-space formalism for describing and analysing the problem, we restrict to the $L^2$-space (which is legitimate by assuming the existence of a classical solution for \eqref{equation}-\eqref{initial}) and fully exploit the Wiener-It\^o-Segal isomorphism between the symmetric Fock space based on a Hilbert space and the family of square integral Brownian functionals.   

To state our main result we introduce a few notations and a couple of standing assumptions. 

\begin{assumption}\label{assumption}
There exists an orthonormal basis $\{\xi_k\}_{k\geq 1}$ of $L^2([0,1])$ that diagonalizes the operator 
\begin{align}\label{A}
\mathcal{A}:=-\partial^2_x+\lambda_d(x),\quad x\in [0,1],
\end{align}
with homogenous Neumann boundary conditions. This means that for all $j,k\geq 1$ we have
\begin{align*}
\int_0^1\xi_k(y)\xi_j(y)dy=\delta_{kj},\quad\xi_k'(0)=\xi_k'(1)=0,
\end{align*}
and there exists a sequence of non negative real numbers $\{\alpha_k\}_{k\geq 1}$ such that
\begin{align*}
\mathcal{A}\xi_k=\alpha_k\xi_k,\quad\mbox{ for all $k\geq 1$}.
\end{align*}
\end{assumption}

\begin{remark}
We note that according to the classical Sturm-Liouville boundary value problem the continuity of $\lambda_d$ is sufficient for Assumption \ref{assumption} to hold true.   
\end{remark}

We now denote by $\Pi_N:L^2([0,1])\to L^2([0,1])$ the orthogonal projection onto the finite dimensional space spanned by $\{\xi_1,...,\xi_N\}$, i.e. 
\begin{align*}
\Pi_Nf(x):=\sum_{k=1}^N\langle f,\xi_k\rangle_{L^2([0,1])}\xi_k(x),\quad x\in [0,1];
\end{align*}
we also set
\begin{align}\label{notation lambda}
d_k:=\langle \lambda_d,\xi_k\rangle_{L^2([0,1])},\quad c_k:=\langle \lambda_c,\xi_k\rangle_{L^2([0,1])},\quad \gamma:=\int_0^1\lambda_c(y)dy,\quad\mbox{ and }\quad\zeta_k:=\langle \zeta,\xi_k\rangle_{L^2([0,1])},
\end{align} 
where the functions $\lambda_d$, $\lambda_c$ and $\zeta$ are those from \eqref{equation}-\eqref{initial}.

\begin{assumption}\label{assumption 2}
There exists $N_0\geq 1$ such that $\Pi_{N_0}\lambda_d=\lambda_d$; this is equivalent to say $\Pi_{N}\lambda_d=\lambda_d$ for all $N\geq N_0$.
\end{assumption}

\begin{remark}
Assumption \eqref{assumption 2} is readily fulfilled in the case of a constant function $\lambda_d(x)=\lambda_d 1(x), x\in [0,1]$; in fact, in this case
\begin{align*}
(\mathcal{A}f)(x)=-f''(x)+\lambda_d f(x),\quad\quad \xi_k(x)=\cos((k-1)\pi x),\quad k\geq 1,
\end{align*} 	
and 
\begin{align*}
\alpha_k=(k-1)^2\pi^2+\lambda_d,\quad k\geq 1.
\end{align*}
This gives
\begin{align*}
\xi_1(x)=1(x)\quad\mbox{ and hence }\quad (\Pi_1\lambda_d)(x)=\lambda_d(x),
\end{align*}
i.e. $N_0=1$.
\end{remark}

We are now ready to state the main result of the present paper; its proof is postponed to Section \ref{proof main result} but a direct verification of its validity is presented in Section \ref{verification} for $n=0,1,2$. In the sequel we set $\Pi_N^{\otimes n}$ to be the orthogonal projection from $L^2([0,1]^n)$ to the linear space generated by the functions $\{\xi_{i_1}\otimes\cdot\cdot\cdot\otimes\xi_{i_n}, 1\leq i_1,...,i_n\leq N\}$.

\begin{theorem}\label{main theorem}
Let Assumptions \ref{assumption}-\ref{assumption 2} be in force and denote by $\{\rho_n\}_{n\geq 0}$ a classical solution of equation \eqref{equation}-\eqref{initial}. Then, for any $N\geq N_0$ and $t\geq 0$ we have the representation
\begin{align}\label{formula n=0}
	\rho_0(t)=\mathbb{E}[u(t,Z)],
\end{align}	
and for any $n\geq 1$ and $(x_1,...,x_n)\in [0,1]^n$,
\begin{align}\label{formula}
\Pi_N^{\otimes n}\rho_n(t,x_1,...,x_n)=\frac{1}{n!}\sum_{j_1,...j_n=1}^N\mathbb{E}\left[\left(\partial_{z_{j_1}}\cdot\cdot\cdot \partial_{z_{j_n}}u\right)(t,Z)\right]\xi_{j_1}(x_1)\cdot\cdot\cdot\xi_{j_n}(x_n).
\end{align}
Here, 
\begin{align}\label{formula-expectation }
\mathbb{E}\left[\left(\partial_{z_{j_1}}\cdot\cdot\cdot \partial_{z_{j_n}}u\right)(t,Z)\right]=\int_{\mathbb{R}^N}\left(\partial_{z_{j_1}}\cdot\cdot\cdot \partial_{z_{j_n}}u\right)(t,z)(2\pi)^{-N/2}e^{-\frac{|z|^2}{2}}dz,
\end{align}	
while $u:[0,+\infty[\times\mathbb{R}^N\to\mathbb{R}$ is a classical solution of the partial differential equation
\begin{align}\label{PDE}
\begin{split}
\partial_tu(t,z)=&\sum_{k=1}^N\alpha_k\partial^2_{z_k}u(t,z)+\sum_{k=1}^N\left(d_k-c_k-\alpha_k z_k\right)\partial_{z_k}u(t,z)+\left(\sum_{k=1}^Nc_kz_k-\gamma\right)u(t,z)\\
u(0,z)=&\sum_{k=1}^N\zeta_kz_k,\quad t\geq 0, z\in\mathbb{R}^N.
\end{split}
\end{align}
\end{theorem}

The paper is organized as follows: in the next section we describe how to verify the validity of formula \eqref{formula} through a direct computation: this should help the reader in understanding the mechanism that relates \eqref{equation}-\eqref{initial} to \eqref{PDE}; Section 3 describes the Gaussian setting needed to formalize our approach: here we recall few basic ideas and tools from Malliavin calculus and infinite dimensional Gaussian analysis; lastly, in Section 4 we prove formula \eqref{formula}, passing through several intermediate steps that illustrate the main ideas of our technique.  

\section{Verification of formula \eqref{formula n=0}-\eqref{formula} for $n=0,1,2$}\label{verification}

The aim of this section is to show via a direct verification the validity of formula \eqref{formula n=0}-\eqref{formula}. This will be done only for $n=0,1,2$ and serves as an illustration of the connection between \eqref{equation}-\eqref{initial} and \eqref{PDE}. First of all, using the notation \eqref{A} we rewrite \eqref{equation} as
\begin{align}\label{equation 2}
	\begin{split}
		\partial_t\rho_0(t)=&\int_0^1\lambda_d(y)\rho_1(t,y)dy-\gamma\rho_0(t);\\
		\partial_t\rho_n(t,x_1,...,x_n)=&-\sum_{i=1}^n\mathcal{A}_i\rho_n(t,x_1,...,x_n)\\
		&+(n+1)\int_0^1\lambda_d(y)\rho_{n+1}(t,x_1,...,x_n,y)dy\\
		&+\frac{1}{n}\sum_{i=1}^n\lambda_c(x_i)\rho_{n-1}(t,x_1,...,x_{i-1},x_{i+1},...,x_n)\\
		&-\gamma\rho_n(t,x_1,...,x_n).
	\end{split}
\end{align}
Then, we find the equation solved by $\{\Pi_N^{\otimes n}\rho_n\}_{n\geq 0}$.
 
\begin{proposition}
Let Assumptions \ref{assumption}-\ref{assumption 2} be in force and denote by $\{\rho_n\}_{n\geq 0}$ a classical solution of equation \eqref{equation 2}-\eqref{initial}. Then, for any $N\geq N_0$ the sequence $\{\Pi_N^{\otimes n}\rho_n\}_{n\geq 0}$ solves
\begin{align}\label{equation projected}
\begin{split}
\partial_t\rho_0(t)=&\int_0^1\lambda_d(y)\Pi_N\rho_1(t,y)dy-\gamma\rho_0(t);\\
\partial_t\Pi_N^{\otimes n}\rho_n(t,x_1,...,x_n)=&-\sum_{j_1,...j_n=1}^N\left(\sum_{i=1}^n\alpha_{j_i}\right)\langle \rho_n(t,\cdot),\xi_{j_1}\otimes\cdot\cdot\cdot\otimes\xi_{j_n}\rangle_{L^2([0,1]^n)}\xi_{j_1}(x_1)\cdot\cdot\cdot\xi_{j_n}(x_n)\\
&+(n+1)\int_0^1\lambda_d(y)\Pi_N^{\otimes (n+1)}\rho_{n+1}(t,x_1,...,x_n,y)dy\\
&+\frac{1}{n}\sum_{i=1}^n\Pi_N\lambda_c(x_i)\Pi_N^{\otimes (n-1)}\rho_{n-1}(t,x_1,...,x_{i-1},x_{i+1},...,x_n)\\
&-\gamma\Pi_N^{\otimes n}\rho_n(t,x_1,...,x_n).
\end{split}
\end{align}
Here, we set $\Pi_N^{\otimes 0}\rho_0:=\rho_0$.
\end{proposition}
	
\begin{proof}
Since $\{\rho_n\}_{n\geq 0}$ solves \eqref{equation 2}, we can write for $n\geq 1$ that
\begin{align}\label{AA}
	\begin{split}
\partial_t\Pi_N^{\otimes n}\rho_n(t,x_1,...,x_n)=&\Pi_N^{\otimes n}\partial_t\rho_n(t,x_1,...,x_n)\\
=&-\Pi_N^{\otimes n}\sum_{i=1}^n\mathcal{A}_i\rho_n(t,x_1,...,x_n)\\
&+(n+1)\int_0^1\lambda_d(y)\Pi_N^{\otimes n}\rho_{n+1}(t,x_1,...,x_n,y)dy\\
&+\frac{1}{n}\sum_{i=1}^n\Pi_N^{\otimes n}\left[\lambda_c(x_i)\rho_{n-1}(t,x_1,...,x_{i-1},x_{i+1},...,x_n)\right]\\
&-\gamma\Pi_N^{\otimes n}\rho_n(t,x_1,...,x_n)\\
=&-\Pi_N^{\otimes n}\sum_{i=1}^n\mathcal{A}_i\rho_n(t,x_1,...,x_n)\\
&+(n+1)\int_0^1\lambda_d(y)\Pi_N^{\otimes n}\rho_{n+1}(t,x_1,...,x_n,y)dy\\
&+\frac{1}{n}\sum_{i=1}^n\Pi_N\lambda_c(x_i)\Pi_N^{\otimes (n-1)}\rho_{n-1}(t,x_1,...,x_{i-1},x_{i+1},...,x_n)\\
&-\gamma\Pi_N^{\otimes n}\rho_n(t,x_1,...,x_n).
\end{split}
\end{align}
We now observe that
\begin{align*}
\sum_{i=1}^n\mathcal{A}_i\rho_n(t,x_1,...,x_n)=&\sum_{i=1}^n\mathcal{A}_i\sum_{j_1,...,j_n\geq 1}\langle \rho_n(t,\cdot),\xi_{j_1}\otimes\cdot\cdot\cdot\otimes\xi_{j_n}\rangle_{L^2([0,1]^n)}\xi_{j_1}(x_1)\cdot\cdot\cdot\xi_{j_n}(x_n)\\
=&\sum_{j_1,...,j_n\geq 1}\left(\sum_{i=1}^n\alpha_{j_i}\right)\langle \rho_n(t,\cdot),\xi_{j_1}\otimes\cdot\cdot\cdot\otimes\xi_{j_n}\rangle_{L^2([0,1]^n)}\xi_{j_1}(x_1)\cdot\cdot\cdot\xi_{j_n}(x_n),
\end{align*}
and hence 
\begin{align*}
&\Pi_N^{\otimes n}\sum_{i=1}^n\mathcal{A}_i\rho_n(t,x_1,...,x_n)\\
&\quad=\sum_{j_1,...,j_n=1}^N\left(\sum_{i=1}^n\alpha_{j_i}\right)\langle \rho_n(t,\cdot),\xi_{j_1}\otimes\cdot\cdot\cdot\otimes\xi_{j_n}\rangle_{L^2([0,1]^n)}\xi_{j_1}(x_1)\cdot\cdot\cdot\xi_{j_n}(x_n).
\end{align*}
Moreover, according to Assumption \ref{assumption 2} we have
\begin{align*}
\int_0^1\lambda_d(y)\Pi_N^{\otimes n}\rho_{n+1}(t,x_1,...,x_n,y)dy=&\int_0^1\Pi_N\lambda_d(y)\Pi_N^{\otimes n}\rho_{n+1}(t,x_1,...,x_n,y)dy\\
=&\int_0^1\lambda_d(y)\Pi_N^{\otimes (n+1)}\rho_{n+1}(t,x_1,...,x_n,y)dy.
\end{align*}
If we employ these two last facts in \eqref{AA}, we arrive at \eqref{equation projected} for $n\geq 1$. Similarly, the equation for $n=0$ can be derived by virtue of Assumption \ref{assumption 2}:  
\begin{align*}
\partial_t\rho_0(t)=&\int_0^1\lambda_d(y)\rho_1(t,y)dy-\gamma\rho_0(t)\\
=&\int_0^1\Pi_N\lambda_d(y)\rho_1(t,y)dy-\gamma\rho_0(t)\\
=&\int_0^1\lambda_d(y)\Pi_N\rho_1(t,y)dy-\gamma\rho_0(t).
\end{align*}  
\end{proof}

\begin{remark}
	The initial and boundary conditions for the sequence $\{\Pi_N^{\otimes n}\rho_n\}_{n\geq 0}$ are easily deduced from \eqref{initial} and the corresponding boundary conditions for $\{\xi_k\}_{k\geq 1}$; more precisely,
	\begin{align}\label{initial projected}
		\begin{split}
			\rho_0(0)&=0;\\
			\Pi_N\rho_1(0,x_1)&=\Pi_N\zeta(x_1),\mbox{ for all $x_1\in[0,1]$};\\
			\Pi_N^{\otimes n}\rho_n(0,x_1,...,x_n)&=0,\quad\mbox{for all $n> 1$ and all $(x_1,...,x_n)\in[0,1]^n$};\\
			\emph{grad }(\Pi_N^{\otimes n}\rho_n)(t,x_1,...,x_n)&=0,\quad\mbox{for all $n\geq 1$, $t\geq 0$ and $(x_1,...,x_n)\in\partial [0,1]^n$}.
		\end{split}
	\end{align}
\end{remark}

\begin{remark}
It is useful to recall that the solution to the Cauchy problem
\begin{align}\label{PDE2}
\begin{split}
\partial_tu(t,z)&=\sum_{i=1}^N\alpha_i\partial^2_{z_i}u(t,z)+\sum_{i=1}^N\left(d_i-c_i-\alpha_i z_i\right)\partial_{z_i}u(t,z)+\left(\sum_{i=1}^Nc_iz_i-\gamma\right)u(t,z);\\
u(0,z)&=\sum_{i=1}^N\zeta_iz_i,
\end{split}
\end{align}
which is the key ingredient of formulas \eqref{formula n=0}-\eqref{formula}, admits the following Feynman-Kac representation (see for instance \cite{KS}):
\begin{align}\label{FK}
u(t,z)=\mathbb{E}\left[\left(\sum_{i=1}^N\zeta_iZ^{z_i}_i(t)\right)\exp\left\{\int_0^t\left(\sum_{i=1}^Nc_iZ^{z_i}_i(s)-\gamma\right)ds\right\}\right],\quad t\geq 0, z=(z_1,...,z_N)\in\mathbb{R}^N.
\end{align} 
Here, for $i\in\{1,...,N\}$, the stochastic process $\{Z_i^{z_i}(t)\}_{t\geq 0}$ is the unique strong solution of the mean-reverting Ornstein-Uhlenbeck stochastic differential equation
\begin{align}\label{OU}
dZ_i^{z_i}(t)=\left(d_i-c_i-\alpha_i Z^{z_i}_i(t)\right)dt+\sqrt{2\alpha_i}dW_i(t),\quad Z_i^{z_i}(0)=z_i, 	
\end{align}
with $\{W_1(t)\}_{t\geq 0}$,..., $\{W_N(t)\}_{t\geq 0}$ being independent one dimensional Brownian motions. It is well known that the solution to \eqref{OU} can be explicitly written for $\alpha_i>0$ as
\begin{align*}
Z^{z_i}_i(t)&=z_ie^{-\alpha_i t}+\frac{d_i-c_i}{\alpha_i}\left(1-e^{-\alpha_i t}\right)+\int_0^te^{-\alpha_i(t-s)}\sqrt{2\alpha_i}dW_i(s),
\end{align*}
and simply 
\begin{align*}
Z^{z_i}_i(t)&=z_i+(d_i-c_1)t,
\end{align*}
when $\alpha_i=0$. This shows that the function $z_i\mapsto Z^{z_i}_i(t)$ is almost surely affine and hence that, according to equation \eqref{FK}, for any $\tau>0$ there exist positive constant $m_1$ and $m_2$ such that
\begin{align}\label{bound}
|u(t,z)|\leq m_1 e^{m_2|z|},\quad\mbox{for all $t\in [0,\tau]$ and $z\in\mathbb{R}^N$}.
\end{align} 
This bound entails the finiteness of the expectation in \eqref{formula n=0}; since the same reasoning applies to the partial spatial derivatives of $u$ (they also satisfy an equation of the form \eqref{PDE2}), we conclude that the expectations in \eqref{formula} are well defined and finite as well. 
\end{remark}

We are now going to verify that the right hand sides of \eqref{formula n=0}-\eqref{formula} solve equation \eqref{equation projected} for $n=0,1,2$; to ease the reference to formula \eqref{formula n=0}-\eqref{formula}, we write down explicitly the first four terms:   
\begin{align}\label{sa}
\begin{split}
\rho_0(t)&=\mathbb{E}[u(t,Z)];\\
\Pi_N\rho_1(t,x_1)&=\sum_{j=1}^N\mathbb{E}\left[(\partial_{z_j}u)(t,Z)\right]\xi_j(x_1);\\
\Pi_N^{\otimes 2}\rho_2(t,x_1,x_2)&=\frac{1}{2}\sum_{j_1,j_2=1}^N\mathbb{E}\left[(\partial_{z_{j_2}}\partial_{z_{j_1}}u)(t,Z)\right]\xi_{j_1}(x_1)\xi_{j_2}(x_2);\\
\Pi_N^{\otimes 3}\rho_3(t,x_1,x_2,x_3)&=\frac{1}{6}\sum_{j_1,j_2,j_3=1}^N\mathbb{E}\left[(\partial_{z_{j_3}}\partial_{z_{j_2}}\partial_{z_{j_1}}u)(t,Z)\right]\xi_{j_1}(x_1)\xi_{j_2}(x_2)\xi_{j_3}(x_3).
\end{split}
\end{align}

\subsection{Equation for  $\rho_0$}

We have:
\begin{align*}
\partial_t\mathbb{E}[u(t,Z)]=&\mathbb{E}[(\partial_tu)(t,Z)]\\
=&\mathbb{E}\left[\sum_{k=1}^N\alpha_k\partial^2_{z_k}u(t,Z)+\sum_{k=1}^N\left(d_k-c_k-\alpha_k Z_k\right)\partial_{z_k}u(t,Z)+\left(\sum_{k=1}^Nc_kZ_k-\gamma\right)u(t,Z)\right]\\
=&\sum_{k=1}^N\alpha_k\mathbb{E}[\partial^2_{z_k}u(t,Z)]+\sum_{k=1}^N\mathbb{E}\left[\left(d_k-c_k-\alpha_k Z_k\right)\partial_{z_k}u(t,Z)\right]\\
&+\mathbb{E}\left[\left(\sum_{k=1}^Nc_kZ_k-\gamma\right)u(t,Z)\right].
\end{align*}
Observe that an integration by parts gives (the boundary term vanishes thanks to the bound \eqref{bound} applied to $\partial_{z_k}u(t,z)$)
\begin{align*}
\mathbb{E}[\partial^2_{z_k}u(t,Z)]=\mathbb{E}[Z_k\partial_{z_k}u(t,Z)];
\end{align*}
therefore,
\begin{align*}
	\partial_t\mathbb{E}[u(t,Z)]=&\sum_{k=1}^N\alpha_k\mathbb{E}[Z_k\partial_{z_k}u(t,Z)]+\sum_{k=1}^N\mathbb{E}\left[\left(d_k-c_k-\alpha_k Z_k\right)\partial_{z_k}u(t,Z)\right]\\
	&+\mathbb{E}\left[\left(\sum_{k=1}^Nc_kZ_k-\gamma\right)u(t,Z)\right]\\
	=&\sum_{k=1}^N\mathbb{E}\left[\left(d_k-c_k\right)\partial_{z_k}u(t,Z)\right]+\mathbb{E}\left[\left(\sum_{k=1}^Nc_kZ_k-\gamma\right)u(t,Z)\right]\\
=&\sum_{k=1}^Nd_k\mathbb{E}\left[\partial_{z_k}u(t,Z)\right]-\sum_{k=1}^Nc_k\mathbb{E}\left[\partial_{z_k}u(t,Z)\right]+\mathbb{E}\left[\left(\sum_{k=1}^Nc_kZ_k-\gamma\right)u(t,Z)\right].
\end{align*}
An additional integration by parts yields
\begin{align*}
\sum_{k=1}^Nc_k\mathbb{E}\left[\partial_{z_k}u(t,Z)\right]=\mathbb{E}\left[\left(\sum_{k=1}^Nc_kZ_k\right)u(t,Z)\right],
\end{align*}
and hence
\begin{align*}
	\partial_t\mathbb{E}[u(t,Z)]=&\sum_{k=1}^Nd_k\mathbb{E}\left[\partial_{z_k}u(t,Z)\right]-\gamma\mathbb{E}\left[u(t,Z)\right]\\
	=&\int_0^1\lambda_d(y)\left(\sum_{j=1}^N\mathbb{E}\left[(\partial_{z_j}u)(t,Z)\right]\xi_j(y)\right)dy-\gamma\mathbb{E}[u(t,Z)].
\end{align*}
This corresponds to equation \eqref{equation projected} for $n=0$, with $\rho_0(t)=\mathbb{E}[u(t,Z)]$ and 
\begin{align*}
\Pi_N\rho_1(t,x_1)=\sum_{j=1}^N\mathbb{E}\left[(\partial_{z_j}u)(t,Z)\right]\xi_j(x_1).
\end{align*}

\subsection{Equation for  $\Pi_N\rho_1$}

We have
\begin{align}\label{n=1 initial}
\partial_t\sum_{j=1}^N\mathbb{E}\left[(\partial_{z_j}u)(t,Z)\right]\xi_j(x_1)=\sum_{j=1}^N\mathbb{E}\left[(\partial_{z_j}\partial_tu)(t,Z)\right]\xi_j(x_1).
\end{align}
Now,
\begin{align*}
\mathbb{E}\left[(\partial_{z_j}\partial_tu)(t,Z)\right]=&\int_{\mathbb{R}^N}(\partial_{z_j}\partial_tu)(t,z)(2\pi)^{-N/2}e^{-\frac{|z|^2}{2}}dz\\
=&\int_{\mathbb{R}^N}\partial_{z_j}\left(\sum_{k=1}^N\alpha_k\partial^2_{z_k}u(t,z)+\sum_{k=1}^N\left(d_k-c_k-\alpha_k z_k\right)\partial_{z_k}u(t,z)\right)(2\pi)^{-N/2}e^{-\frac{|z|^2}{2}}dz\\
&+\int_{\mathbb{R}^N}\partial_{z_j}\left(\left(\sum_{k=1}^Nc_kz_k-\gamma\right)u(t,z)\right)(2\pi)^{-N/2}e^{-\frac{|z|^2}{2}}dz\\
=&\int_{\mathbb{R}^N}\left(\sum_{k=1}^N\alpha_k\partial_{z_j}\partial^2_{z_k}u(t,z)-\alpha_j\partial_{z_j}u(t,z)\right)(2\pi)^{-N/2}e^{-\frac{|z|^2}{2}}dz\\
&+\int_{\mathbb{R}^N}\left(\sum_{k=1}^N\left(d_k-c_k-\alpha_k z_k\right)\partial_{z_j}\partial_{z_k}u(t,z)\right)(2\pi)^{-N/2}e^{-\frac{|z|^2}{2}}dz\\
&+\int_{\mathbb{R}^N}\left(c_ju(t,z)+\left(\sum_{k=1}^Nc_kz_k-\gamma\right)\partial_{z_j}u(t,z)\right)(2\pi)^{-N/2}e^{-\frac{|z|^2}{2}}dz.
\end{align*}
An integration by parts with respect to $\partial_{z_k}$ in the first term of the fourth line above will produce the term 
\begin{align*}
\int_{\mathbb{R}^N}\left(\sum_{k=1}^N\alpha_k z_k\partial_{z_j}\partial_{z_k}u(t,z)\right)(2\pi)^{-N/2}e^{-\frac{|z|^2}{2}}dz,
\end{align*}
which is identical to one of the terms from the fifth line but opposite in sign. We can therefore write
\begin{align}
\begin{split}\label{n=1}
\mathbb{E}\left[(\partial_{z_j}\partial_tu)(t,Z)\right]=&\int_{\mathbb{R}^N}\left(-\alpha_j\partial_{z_j}u(t,z)+\sum_{k=1}^N\left(d_k-c_k\right)\partial_{z_j}\partial_{z_k}u(t,z)\right)(2\pi)^{-N/2}e^{-\frac{|z|^2}{2}}dz\\
&+\int_{\mathbb{R}^N}\left(c_ju(t,z)+\left(\sum_{k=1}^Nc_kz_k-\gamma\right)\partial_{z_j}u(t,z)\right)(2\pi)^{-N/2}e^{-\frac{|z|^2}{2}}dz.
\end{split}
\end{align}
Similarly, an integration by parts with respect to $\partial_{z_k}$ in the term
\begin{align*}
\int_{\mathbb{R}^N}\left(-\sum_{k=1}^Nc_k\partial_{z_j}\partial_{z_k}u(t,z)\right)(2\pi)^{-N/2}e^{-\frac{|z|^2}{2}}dz
\end{align*}
from the first line of \eqref{n=1} will give
\begin{align*}
	-\int_{\mathbb{R}^N}\left(\sum_{k=1}^Nc_kz_k\right)\partial_{z_j}u(t,z)(2\pi)^{-N/2}e^{-\frac{|z|^2}{2}}dz.
\end{align*}
hence cancelling the corresponding term from the second line in \eqref{n=1}. Summing up,
\begin{align*}
\mathbb{E}\left[(\partial_{z_j}\partial_tu)(t,Z)\right]=&\int_{\mathbb{R}^N}\left(-\alpha_j\partial_{z_j}u(t,z)+\sum_{k=1}^Nd_k\partial_{z_j}\partial_{z_k}u(t,z)\right)(2\pi)^{-N/2}e^{-\frac{|z|^2}{2}}dz\\
&+\int_{\mathbb{R}^N}\left(c_ju(t,z)-\gamma\partial_{z_j}u(t,z)\right)(2\pi)^{-N/2}e^{-\frac{|z|^2}{2}}dz\\
=&-\alpha_j\mathbb{E}[(\partial_{z_j}u)(t,Z)]+\sum_{k=1}^Nd_k\mathbb{E}\left[(\partial_{z_j}\partial_{z_k}u)(t,Z)\right]\\
&+c_j\mathbb{E}[u(t,Z)]-\gamma\mathbb{E}[(\partial_{z_j}u)(t,Z)].
\end{align*} 
We can now plug the last expression in \eqref{n=1 initial} to get
\begin{align*}
\partial_t\sum_{j=1}^N\mathbb{E}\left[(\partial_{z_j}u)(t,Z)\right]\xi_j(x_1)=&\sum_{j=1}^N\mathbb{E}\left[(\partial_{z_j}\partial_tu)(t,Z)\right]\xi_j(x_1)\\
=&-\sum_{j=1}^N\alpha_j\mathbb{E}[(\partial_{z_j}u)(t,Z)]\xi_j(x_1)+\sum_{j=1}^N\sum_{k=1}^Nd_k\mathbb{E}\left[(\partial_{z_j}\partial_{z_k}u)(t,Z)\right]\xi_j(x_1)\\
&+\sum_{j=1}^Nc_j\mathbb{E}[u(t,Z)]\xi_j(x_1)-\gamma\sum_{j=1}^N\mathbb{E}[(\partial_{z_j}u)(t,Z)]\xi_j(x_1)\\
=&-\sum_{j=1}^N\alpha_j\mathbb{E}[(\partial_{z_j}u)(t,Z)]\xi_j(x_1)+2\int_0^1\lambda_d(y)\Pi_N^{\otimes 2}\rho_2(t,x_1,y)dy\\
&+\Pi_N\lambda_c(x_1)\rho_0(t)-\gamma\Pi_N\rho_1(t,x_1).
\end{align*}
This corresponds to equation \eqref{equation projected} for $n=1$ with the prescriptions \eqref{sa}. 

\subsection{Equation for  $\Pi_N^{\otimes 2}\rho_2$}

We start as before with
\begin{align}
\begin{split}\label{n=2}
&\partial_t\left(\frac{1}{2}\sum_{j_1,j_2=1}^N\mathbb{E}\left[(\partial_{z_{j_2}}\partial_{z_{j_1}}u)(t,Z)\right]\xi_{j_1}(x_1)\xi_{j_2}(x_2)\right)\\
&=\quad\frac{1}{2}\sum_{j_1,j_2=1}^N\mathbb{E}\left[(\partial_{z_{j_2}}\partial_{z_{j_1}}\partial_tu)(t,Z)\right]\xi_{j_1}(x_1)\xi_{j_2}(x_2).
\end{split}
\end{align}
Now,
\begin{align*}
&\mathbb{E}\left[(\partial_{z_{j_2}}\partial_{z_{j_1}}\partial_tu)(t,Z)\right]\\
&\quad=\int_{\mathbb{R}^N}\partial_{z_{j_2}}\partial_{z_{j_1}}\partial_tu(t,z)(2\pi)^{-\frac{N}{2}}e^{-\frac{|z|^2}{2}}dz\\
&\quad=\int_{\mathbb{R}^N}\partial_{z_{j_2}}\partial_{z_{j_1}}\left(\sum_{k=1}^N\alpha_k\partial^2_{z_k}u(t,z)+\sum_{k=1}^N\left(d_k-c_k-\alpha_k z_k\right)\partial_{z_k}u(t,z)\right)(2\pi)^{-\frac{N}{2}}e^{-\frac{|z|^2}{2}}dz\\
&\quad\quad+\int_{\mathbb{R}^N}\partial_{z_{j_2}}\partial_{z_{j_1}}\left(\left(\sum_{k=1}^Nc_kz_k-\gamma\right)u(t,z)\right)(2\pi)^{-\frac{N}{2}}e^{-\frac{|z|^2}{2}}dz\\
&\quad=\sum_{k=1}^N\alpha_k\int_{\mathbb{R}^N}\partial_{z_{j_2}}\partial_{z_{j_1}}\partial^2_{z_k}u(t,z)(2\pi)^{-\frac{N}{2}}e^{-\frac{|z|^2}{2}}dz\\
&\quad\quad+\int_{\mathbb{R}^N}\left(-\alpha_{j_1}\partial_{z_{j_2}}\partial_{z_{j_1}}u(t,z)-\alpha_{j_2}\partial_{z_{j_1}}\partial_{z_{j_2}}u(t,z)\right)(2\pi)^{-\frac{N}{2}}e^{-\frac{|z|^2}{2}}dz\\
&\quad\quad+\int_{\mathbb{R}^N}\left(\sum_{k=1}^N\left(d_k-c_k-\alpha_k z_k\right)\partial_{z_{j_2}}\partial_{z_{j_1}}\partial_{z_k}u(t,z)\right)(2\pi)^{-\frac{N}{2}}e^{-\frac{|z|^2}{2}}dz\\
&\quad\quad+\int_{\mathbb{R}^N}\left(c_{j_1}\partial_{z_{j_2}}u(t,z)+c_{j_2}\partial_{z_{j_1}}u(t,z)+\left(\sum_{k=1}^Nc_kz_k-\gamma\right)\partial_{z_{j_2}}\partial_{z_{j_1}}u(t,z)\right)(2\pi)^{-\frac{N}{2}}e^{-\frac{|z|^2}{2}}dz.
\end{align*}
Integration by parts with respect to $\partial_{z_k}$ in 
\begin{align*}
\int_{\mathbb{R}^N}\partial_{z_{j_2}}\partial_{z_{j_1}}\partial^2_{z_k}u(t,z)(2\pi)^{-\frac{N}{2}}e^{-\frac{|z|^2}{2}}dz
\end{align*}
and 
\begin{align*}
\int_{\mathbb{R}^N}\left(-\sum_{k=1}^Nc_k\partial_{z_{j_2}}\partial_{z_{j_1}}\partial_{z_k}u(t,z)\right)(2\pi)^{-\frac{N}{2}}e^{-\frac{|z|^2}{2}}dz
\end{align*}
will simplify the expression to 
\begin{align*}
&\mathbb{E}\left[(\partial_{z_{j_2}}\partial_{z_{j_1}}\partial_tu)(t,Z)\right]\\
&\quad=\int_{\mathbb{R}^N}\left(-\alpha_{j_1}\partial_{z_{j_2}}\partial_{z_{j_1}}u(t,z)-\alpha_{j_2}\partial_{z_{j_1}}\partial_{z_{j_2}}u(t,z)\right)(2\pi)^{-\frac{N}{2}}e^{-\frac{|z|^2}{2}}dz\\
&\quad\quad+\sum_{k=1}^Nd_k\int_{\mathbb{R}^N}\partial_{z_{j_2}}\partial_{z_{j_1}}\partial_{z_k}u(t,z)(2\pi)^{-\frac{N}{2}}e^{-\frac{|z|^2}{2}}dz\\
&\quad\quad+\int_{\mathbb{R}^N}\left(c_{j_1}\partial_{z_{j_2}}u(t,z)+c_{j_2}\partial_{z_{j_1}}u(t,z)-\gamma\partial_{z_{j_2}}\partial_{z_{j_1}}u(t,z)\right)(2\pi)^{-\frac{N}{2}}e^{-\frac{|z|^2}{2}}dz\\
&\quad=-\left(\alpha_{j_1}+\alpha_{j_2}\right)\mathbb{E}\left[\left(\partial_{z_{j_2}}\partial_{z_{j_1}}u\right)(t,Z)\right]+\sum_{k=1}^Nd_k\mathbb{E}\left[\left(\partial_{z_{j_2}}\partial_{z_{j_1}}\partial_{z_k}u\right)(t,Z)\right]\\
&\quad\quad+c_{j_1}\mathbb{E}\left[(\partial_{z_{j_2}}u)(t,Z)\right]+c_{j_2}\mathbb{E}\left[(\partial_{z_{j_1}}u)(t,Z)\right]-\gamma\mathbb{E}\left[\left(\partial_{z_{j_2}}\partial_{z_{j_1}}u\right)(t,Z)\right]
\end{align*}
Therefore,
\begin{align*}
\partial_t\Pi_N^{\otimes 2}\rho_2(t,x_1,x_2)
=&\frac{1}{2}\sum_{j_1,j_2=1}^N\mathbb{E}\left[(\partial_{z_{j_2}}\partial_{z_{j_1}}\partial_tu)(t,Z)\right]\xi_{j_1}(x_1)\xi_{j_2}(x_2)\\
=&-\frac{1}{2}\sum_{j_1,j_2=1}^N\left(\alpha_{j_1}+\alpha_{j_2}\right)\mathbb{E}\left[\left(\partial_{z_{j_2}}\partial_{z_{j_1}}u\right)(t,Z)\right]\xi_{j_1}(x_1)\xi_{j_2}(x_2)\\
&+\frac{1}{2}\sum_{j_1,j_2=1}^N\sum_{k=1}^Nd_k\mathbb{E}\left[\left(\partial_{z_{j_2}}\partial_{z_{j_1}}\partial_{z_k}u\right)(t,Z)\right]\xi_{j_1}(x_1)\xi_{j_2}(x_2)\\
&+\frac{1}{2}\sum_{j_1,j_2=1}^N \left(c_{j_1}\mathbb{E}\left[(\partial_{z_{j_2}}u)(t,Z)\right]+c_{j_2}\mathbb{E}\left[(\partial_{z_{j_1}}u)(t,Z)\right]\right)\xi_{j_1}(x_1)\xi_{j_2}(x_2)\\
&-\frac{\gamma}{2}\sum_{j_1,j_2=1}^N\mathbb{E}\left[\left(\partial_{z_{j_2}}\partial_{z_{j_1}}u\right)(t,Z)\right]\xi_{j_1}(x_1)\xi_{j_2}(x_2)\\
=&-\frac{1}{2}\sum_{j_1,j_2=1}^N\left(\alpha_{j_1}+\alpha_{j_2}\right)\mathbb{E}\left[\left(\partial_{z_{j_2}}\partial_{z_{j_1}}u\right)(t,Z)\right]\xi_{j_1}(x_1)\xi_{j_2}(x_2)\\
&+3\int_0^1\lambda_d(y)\Pi_N^{\otimes 3}\rho_3(t,x_1,x_2,y)dy+(\Pi_N\lambda_c\hat{\otimes}\Pi_N\rho_1(t,\cdot))(x_1,x_2)\\
&-\gamma\Pi_N^{\otimes 2}\rho_2(t,x_1,x_2).
\end{align*}
This corresponds through identities \eqref{sa} to equation \eqref{equation projected} for $n=2$.

\section{Preliminary material}

In this section we introduce the framework utilized for proving our main theorem. For more details on these topics, we refer the reader to one of the books \cite{Bogachev}, \cite{Janson} and \cite{Nualart}.

\subsection{Wiener-It\^o chaos expansion}
Let $(\Omega,\mathcal{B},\mathbb{P})$ be the classical Wiener space over the interval $[0,1]$, i.e. $\Omega$ is the space of continuous functions defined on the interval $[0,1]$ and null at zero, $\mathcal{B}$ is the Borel $\sigma$-algebra of $\Omega$ induced by the supremum norm and $\mathbb{P}$ the Wiener measure on $(\Omega,\mathcal{B})$. We denote by 
\begin{align*}
\begin{split}
B_x&:\Omega\to\mathbb{R}\\
&\quad\omega\mapsto B_x(\omega):=\omega(x),\quad x\in [0,1],
\end{split}
\end{align*} 
the coordinate process which by construction is a one dimensional Brownian motion under $\mathbb{P}$. According to the Wiener-It\^o chaos expansion theorem, any random variable $\Phi$ in $\mathbb{L}^2(\Omega)$ can be uniquely represented as 
\begin{align}\label{WI}
\Phi=\sum_{n\geq 0}I_n(h_n),
\end{align}
where 
\begin{itemize}
\item $h_0:=\mathbb{E}[\Phi]$;
\item for $n\geq 1$, $h_n\in L_{s}^2([0,1]^n)$, the space of square integrable symmetric functions;
\item $I_0(h_0):=h_0=\mathbb{E}[\Phi]$;
\item for $n\geq 1$, $I_n(h_n)$ stands for the $n$-th order multiple It\^o integral defined as 
\begin{align*}
I_n(h_n):=n!\int_0^1\int_0^{x_1}\cdot\cdot\cdot\int_0^{x_{n-1}}h_n(x_1,...,x_n)dB_{x_n}\cdot\cdot\cdot dB_{x_2}dB_{x_1}.
\end{align*}
\end{itemize}
The series in \eqref{WI} provides an orthogonal decomposition of $\Phi$ that converges in $\mathbb{L}^2(\Omega)$; in fact, multiple It\^o integrals possess the following general properties:  
\begin{itemize}
	\item for all $n\geq 1$, $\mathbb{E}[I_n(h_n)]=0$;
	\item if $n\neq m$, then $\mathbb{E}[I_n(h_n)I_m(h_m)]=0$;
	\item for $n\geq 1$, $\mathbb{E}[I_n(h_n)^2]=n!|h_n|^2_{L^2([0,1]^n)}$.
\end{itemize}
From the last two identities we get
\begin{align*}
\mathbb{E}[\Phi^2]=\sum_{n\geq 0}n!|h_n|^2_{L^2([0,1]^n)},
\end{align*}
and, for $\Psi\in\mathbb{L}^2(\Omega)$ with
\begin{align}\label{WI Psi}
\Psi=\sum_{n\geq 0}I_n(g_n),
\end{align} 
that
\begin{align}\label{scalar product}
\mathbb{E}[\Phi \Psi]=\sum_{n\geq 0}n!\langle h_n,g_n\rangle_{L^2([0,1]^n)}.
\end{align}
Two notable subsets of $\mathbb{L}^2(\Omega)$ are
\begin{align*}
\mathtt{F}:=\left\{\sum_{n=0}^MI_n(h_n),\mbox{ for some $M\in\mathbb{N}\cup\{0\}$, $h_0\in\mathbb{R}$ and $h_n\in L_s^2([0,1]^n)$, $n=1,...,M$}\right\}, 
\end{align*}
which collects the random variables with a finite order chaos expansion, and 
\begin{align*}
\mathtt{E}:=\left\{\mathcal{E}(f):=\sum_{n\geq 0}I_n\left(\frac{f^{\otimes n}}{n!}\right),\mbox{ for some $f\in L^2([0,1])$}\right\},
\end{align*}
which is the family of the so-called \emph{stochastic exponentials}. It is well known that
\begin{align*}
\mathcal{E}(f)=\exp\left\{I_1(f)-\frac{1}{2}|f|^2_{L^2([0,1])}\right\}
\end{align*} 
and that $\mathtt{F}$ and the linear span of $\mathtt{E}$ are both dense in $\mathbb{L}^2(\Omega)$. In particular,
\begin{align*}
\mathbb{E}[\Phi Z]=\mathbb{E}[\Psi Z],\quad\mbox{for all $Z\in \mathtt{F}$},
\end{align*}
or
\begin{align*}
\mathbb{E}[\Phi\mathcal{E}(f)]=\mathbb{E}[\Psi\mathcal{E}(f)],\quad\mbox{for all $f\in L^2([0,1])$ (or some dense subset of $L^2([0,1])$)},
\end{align*}
implies $\Phi=\Psi$, $\mathbb{P}$-a.s.. Note also that, according to \eqref{scalar product}, we can write
\begin{align*}
\mathbb{E}[\Phi\mathcal{E}(f)]=\sum_{n\geq 0}\langle h_n,f^{\otimes n}\rangle_{L^2([0,1]^n)},
\end{align*}
whenever $\Phi=\sum_{n\geq 0}I_n(h_n)$. We recall in addition that, by virtue of the \emph{Hu-Meyer formula}
\begin{align}\label{HM}
I_n(h_n)\cdot I_m(h_m)=\sum_{r=0}^{n\wedge n}r!{n \choose r}{m \choose r}I_{n+m-2r}(h_n\hat{\otimes}_r h_m),
\end{align} 
the linear space $\mathtt{F}$ is an algebra with respect to the point-wise multiplication. Here, $h_n\otimes_r h_m$ stands for the \emph{$r$-th order contraction} of $h_n$ and $h_m$, i.e.
\begin{align}\label{contraction}
\begin{split}
&(h_n\otimes_r h_m)(x_1,....,x_{n+m-2r})\\
&:=\int_{[0,1]^r}h_n(x_1,...,x_{n-r},y_1,...,y_r)h_m(y_1,...,y_r,x_{n-r+1},...,x_{n+m-2r})dy_1\cdot\cdot\cdot dy_r,
\end{split}
\end{align}
while $h_n\hat{\otimes}_r h_m$ denotes the \emph{symmetrization} of $h_n\otimes_r h_m$, i.e.
\begin{align}\label{symm}
	\begin{split}
	&(h_n\hat{\otimes}_r h_m)(x_1,....,x_{n+m-2r})\\
	&:=\frac{1}{(n+m-2r)!}\sum_{\sigma\in S_{n+m-2r}}(h_n\otimes_r h_m)(x_{\sigma(1)},...,x_{\sigma(n+m-2r)}),
\end{split}
\end{align}
with  $S_{n+m-2r}$ being the group of permutations on $\{1,...,n+m-2r\}$.

\subsection{Malliavin derivative}

The \emph{Malliavin derivative} of $\Phi=\sum_{n=0}^MI_n(h_n)\in\mathtt{F}$, denoted $\{D_x\Phi\}_{x\in[0,1]}$, is the element of $L^2([0,1];\mathtt{F})$ defined by
\begin{align*}
D_x\Phi:=\sum_{n=0}^{M-1}(n+1)I_n(h_{n+1}(\cdot,x)), \quad x\in [0,1].
\end{align*}
For $l\in L^2([0,1])$ and $\Phi=\sum_{n=0}^MI_n(h_n)\in\mathtt{F}$, we also write
\begin{align}
\begin{split}
D_l\Phi:=\langle D\Phi,l\rangle_{L^2([0,1])}&=\sum_{n=0}^{M-1}(n+1)I_{n}\left(\int_0^1h_{n+1}(\cdot,y)l(y)dy\right)\\
&=\sum_{n=0}^{M-1}(n+1)I_{n}\left(h_{n+1}\otimes_1 l\right)
\end{split}
\end{align}
for the \emph{directional Malliavin derivative} of $\Phi$ along $l$ (in the last member above we utilized the notation \eqref{contraction}). We remark that $D_l\Phi$ is also a member of $\mathtt{F}$.\\
If we now take $l\in L^2([0,1])$, $\Phi=\sum_{n=0}^MI_n(h_n)\in\mathtt{F}$ and $\Psi=\sum_{n=0}^KI_n(g_n)\in\mathtt{F}$, we can write
\begin{align}\label{wick}
\mathbb{E}[D_l\Phi\cdot\Psi]&=\sum_{n=0}^{(M-1)\wedge K}n!(n+1)\langle h_{n+1}\otimes_1 l,g_n\rangle_{L^2([0,1]^n)}\nonumber\\
&=\sum_{n=0}^{(M-1)\wedge K}(n+1)!\langle h_{n+1},l\otimes g_n\rangle_{L^2([0,1]^{n+1})}\nonumber\\
&=\sum_{n=0}^{(M-1)\wedge K}(n+1)!\langle h_{n+1},l\hat{\otimes}g_n\rangle_{L^2([0,1]^{n+1})}\\
&=\sum_{n=1}^{M\wedge (K+1)}n!\langle h_{n},l\hat{\otimes}g_{n-1}\rangle_{L^2([0,1]^{n})}\nonumber\\
&=\mathbb{E}[\Phi\cdot D_l^{\star}\Psi]&\nonumber,
\end{align}
where
\begin{align*}
D^{\star}_l\Psi:=\sum_{n=1}^{K+1}I_n(l\hat{\otimes}g_{n-1})
\end{align*}
and
\begin{align*}
(l\hat{\otimes}g_{n-1})(x_1,...,x_n):=\frac{1}{n}\sum_{i=1}^nf(x_i)g_{n-1}(x_1,...,x_{i-1},x_{i+1},...,x_n)
\end{align*}
(compare with definitions \eqref{contraction} and \eqref{symm} for $r=0$). We remark that in \eqref{wick} we utilized the symmetry of $h_n$ and the fact that the symmetrization operator is idempotent and self-adjoint in $L^2([0,1]^n)$. \\
It is clear  that $D^{\star}_l\Psi$ also belongs to $\mathtt{F}$; moreover, $\mathbb{E}[D^{\star}_l\Psi]=0$, for all $l\in L^2([0,1])$ and $\Psi\in\mathtt{F}$.\\
If in the Hu-Meyer formula \eqref{HM} we take $m=1$ and set $g_1=l\in L^2([0,1])$, we get
\begin{align*}
I_n(h_n)\cdot I_1(l)&=I_{n+1}(h_n\hat{\otimes}l)+I_{n-1}(h_n\otimes_1 l)\\
&=D_{l}^{\star}I_n(h_n)+D_{l}I_n(h_n).
\end{align*}
Summing over $n$ and using the linearity of the operators $D^{\star}_l$ and $D_l$, we deduce the identity
\begin{align}\label{gaussian commutation}
D_l^{\star}\Psi+D_l\Psi=\Psi\cdot I_1(l),
\end{align}
which is valid for $\Psi\in\mathtt{F}$ and $l\in L^2([0,1])$. One can also introduce the adjoint of $D_x$, denoted $\delta$: if $\Phi(x)=\sum_{n=0}^MI_n(h_n(\cdot;x))$ is a stochastic process in $\mathtt{F}$, then
\begin{align*}
\delta(\Phi(\cdot)):=\sum_{n=0}^MI_{n+1}(\tilde{h}_n)\in\mathtt{F},
\end{align*}
where $\tilde{h}_n$ stands for the symmetrization of $h_n$ with respect to the $n+1$ variables $x_1,...,x_n,x$.\\
It is useful to mention that the definition of Malliavin derivative can also be extended to the members of the family $\mathtt{E}$: for any $f,l\in L^2([0,1])$, we have
\begin{align}\label{Malliavin for exp}
D_x\mathcal{E}(f)=f(x)\mathcal{E}(f), x\in [0,1]\quad\mbox{ and }\quad D_l\mathcal{E}(f)=\langle f,l\rangle_{L^2([0,1])}\mathcal{E}(f).
\end{align} 

\subsection{Second quantization operators}

Let $A:L^2([0,1])\to L^2([0,1])$ be a bounded linear operator; for $\Phi=\sum_{n=0}^MI_n(h_n)\in\mathtt{F}$ we define the \emph{second quantization operator} of $A$ as
\begin{align*}
\Gamma(A)\Phi:=\sum_{n=0}^MI_n\left(A^{\otimes n}h_n\right),
\end{align*}
and the \emph{differential second quantization operator} of $A$ as
\begin{align*}
d\Gamma(A)\Phi:=\sum_{n=1}^MI_n\left(\sum_{i=1}^nA_ih_n\right),
\end{align*}
where $A_i$ stands for the operator $A$ acting on the $i$-th variable of $h_n$. The boundedness of $A$ implies that both $\Gamma(A)\Phi$ and $d\Gamma(A)\Psi$ also belong to $\mathtt{F}$; note in addition that for $A$ being the identity, we recover from $d\Gamma(A)$ the well known \emph{number operator}:
\begin{align*}
\mathcal{N}\Phi=\sum_{n=1}^MnI_n\left(h_n\right).
\end{align*}
Via a simple verification on can see that for all $\Phi$ and $\Psi$ in $\mathtt{F}$ the following identities hold true:
\begin{align*}
	\begin{split}
&\mathbb{E}[\Gamma(A)\Phi]=\mathbb{E}[\Phi];\quad \mathbb{E}[d\Gamma(A)\Phi]=0;\\
&\mathbb{E}[\Gamma(A)\Phi\cdot\Psi]=\mathbb{E}[\Phi\cdot\Gamma(A^{\star})\Psi];\quad\mathbb{E}[d\Gamma(A)\Phi\cdot\Psi]=\mathbb{E}[\Phi\cdot d\Gamma(A^{\star})\Psi]. 
\end{split}
\end{align*}
Here, $A^{\star}$ denotes the adjoint of $A$ in $L^2([0,1])$. As for the Malliavin derivative, the actions of second quantization and differential second quantization operators can be extended to the class $\mathtt{E}$ of stochastic exponentials:
\begin{align*}
\Gamma(A)\mathcal{E}(f)=\mathcal{E}(Af)\quad\mbox{ and }\quad d\Gamma(A)\mathcal{E}(f)=D^{\star}_{Af}\mathcal{E}(f).
\end{align*}   
The differential second quantization operator can also be represented as a composition of the Malliavin derivative and its adjoint; more precisely,
\begin{align}\label{dsq}
d\Gamma(A)\Phi=\delta\left(AD_{\cdot}\Phi\right).
\end{align} 

\subsection{Space of generalized random variables}

We remark that the convergence of the series in \eqref{scalar product} is implied via the Cauchy-Schwartz inequality by the conditions
\begin{align}\label{finite series}
\sum_{n\geq 0}n!|h_n|^2_{L^2([0,1]^n)}<+\infty\quad\mbox { and }\quad\sum_{n\geq 0}n!|g_n|^2_{L^2([0,1]^n)}<+\infty.
\end{align}
However, if $\Psi$ has a finite order expansion, i.e.
\begin{align}\label{finite Psi}
\Psi=\sum_{n=0}^MI_n(g_n),\mbox{ for some $M\in\mathbb{N}\cup\{0\}$},
\end{align}
then one can drop the first condition in \eqref{finite series} and get a still well-defined pairing between the \emph{generalized} random variable $\Phi$, represented by the formal series $\sum_{n\geq 0}I_n(h_n)$ (which in general will not convergence in $\mathbb{L}^2(\Omega)$), and the \emph{regular} or \emph{test} random variable $\Psi$ with finite order expansion \eqref{finite Psi}. Let  
\begin{align*}
\mathtt{F}^{\star}:=\left\{\sum_{n\geq 0}I_n(h_n),\mbox{ for some $h_0\in\mathbb{R}$ and $h_n\in L_s^2([0,1]^n)$, $n\geq 1$}\right\} 
\end{align*}
be a family of \emph{generalized} random variables. The action of $T=\sum_{n\geq 0}I_n(h_n)\in\mathtt{F}^{\star}$ on $\varphi=\sum_{n=0}^MI_n(g_n)\in\mathtt{F}$ is defined as 
\begin{align*}
\langle\langle T,\varphi\rangle\rangle:=\sum_{n=0}^Mn!\langle h_n,g_n\rangle_{L^2([0,1]^n)}.
\end{align*}  
By construction, we have the inclusions
\begin{align*}
\mathtt{F}\subset\mathbb{L}^2(\Omega)\subset\mathtt{F}^{\star}
\end{align*}
with
\begin{align*}
\langle\langle T,\varphi\rangle\rangle=\mathbb{E}[T\varphi],
\end{align*}
whenever $T\in\mathbb{L}^2(\Omega)$. We will say that $T=U$ in $\mathtt{F}^{\star}$ if
\begin{align*}
\langle\langle T,\varphi\rangle\rangle=\langle\langle U,\varphi\rangle\rangle, \quad \mbox{for all }\varphi\in\mathtt{F}.
\end{align*}
The \emph{generalized expectation} of $T=\sum_{n\geq 0}I_n(h_n)\in\mathtt{F}^{\star}$ is $\mathbb{E}[T]:=\langle\langle T,1\rangle\rangle=h_0$. It is also important to observe that, according to the Hu-Meyer formula \eqref{HM}, the vector space $\mathtt{F}$ is closed with respect to the point-wise product between random variables. Therefore, if $T\in\mathtt{F}^{\star}$ and $\psi\in\mathtt{F}$, the product $T\cdot\psi$ is well defined and corresponds to the element of $\mathtt{F}^{\star}$ given by the prescription
\begin{align*}
\langle\langle T\cdot\psi,\varphi\rangle\rangle:=\langle\langle T,\psi\cdot\varphi\rangle\rangle, \quad\varphi\in\mathtt{F}.
\end{align*}
The definitions of Malliavin derivative, its adjoint and (differential) second quantization operators can be lifted from $\mathtt{F}$ to $\mathtt{F}^{\star}$ by duality:
\begin{align*}
&\langle\langle D_lT,\varphi\rangle\rangle:=\langle\langle T,D^{\star}_l\varphi\rangle\rangle,\quad \langle\langle D^{\star}_lT,\varphi\rangle\rangle:=\langle\langle T,D_l\varphi\rangle\rangle\\
&\langle\langle \Gamma(A)T,\varphi\rangle\rangle:=\langle\langle T,\Gamma(A^{\star})\varphi\rangle\rangle,\quad \langle\langle d\Gamma(A)T,\varphi\rangle\rangle:=\langle\langle T,d\Gamma(A^{\star})\varphi\rangle\rangle.
\end{align*}
Lastly, we recall a generalized version of the so-called Stroock-Taylor formula: if $T=\sum_{n\geq 0}I_n(h_n)\in\mathtt{F}^{\star}$, then
\begin{align}\label{ST}
h_n(x_1,...,x_n)=\frac{1}{n!}\mathbb{E}[D_{x_1,...,x_n}T], \quad (x_1,...,x_n)\in [0,1]^n.
\end{align}
Here, $\mathbb{E}[D_{x_1,...,x_n}T]$ stands for the the generalized expectation of the $n$-th order Malliavin derivative of $T$.

\begin{remark}
	The space $\mathtt{F}^{\star}$ has been already utilized for solving some stochastic partial differential equations which admits only generalized solutions. See for instance the paper \cite{MR} where the authors investigate an unbiased version of the stochastic Navier-Stokes equations.	
\end{remark}

\section{Proof of Theorem \ref{main theorem}}\label{proof main result}

In this section we present the rigorous derivation of formula \eqref{formula n=0}-\eqref{formula}. Firstly, we devote our attention to the original system, i.e. the one without projection operators $\Pi_N^{\otimes n}$, that we report here for easiness of reference:
\begin{align}\label{equation proof}
	\begin{split}
		\partial_t\rho_0(t)=&\int_0^1\lambda_d(y)\rho_1(t,y)dy-\gamma\rho_0(t);\\
		\partial_t\rho_n(t,x_1,...,x_n)=&-\sum_{i=1}^n\mathcal{A}_i\rho_n(t,x_1,...,x_n)\\
		&+(n+1)\int_0^1\lambda_d(y)\rho_{n+1}(t,x_1,...,x_n,y)dy\\
		&+\frac{1}{n}\sum_{i=1}^n\lambda_c(x_i)\rho_{n-1}(t,x_1,...,x_{i-1},x_{i+1},...,x_n)\\
		&-\gamma\rho_n(t,x_1,...,x_n),
	\end{split}
\end{align}
with initial and boundary conditions
\begin{align}\label{initial proof}
	\begin{split}
	\rho_0(0)&=0;\\
	\rho_1(0,x_1)&=\zeta(x_1),\quad x_1\in[0,1];\\
	\rho_n(0,x_1,...,x_n)&=0,\quad n> 1, (x_1,...,x_n)\in[0,1]^n;\\
	\partial_{\nu}\rho_n(t,x_1,...,x_n)&=0,\quad n\geq 1, t\geq 0, (x_1,...,x_n)\in\partial [0,1]^n.
	\end{split}
\end{align}

The first fundamental step of our analysis consists in integrating all the spatial variables of $\rho_n(t,x_1,...,x_n)$ with respect to the one dimensional Brownian motion $\{B_x\}_{x\in[0,1]}$; this procedure will produce a sequence of time-dependent multiple It\^o integrals which satisfies a stochastic counterpart of equation \eqref{equation proof}-\eqref{initial proof}.   

\begin{proposition}\label{proposition 1}
	Let $\{\rho_n\}_{n\geq 0}$ be a classical solution to \eqref{equation proof}-\eqref{initial proof}. Then, the sequence of random variables $\{I_n(\rho_n(t,\cdot))\}_{n\geq 0}$ satisfies the equations
	\begin{align}\label{equation 5}
	\begin{split}
	\partial_tI_n(\rho_n(t,\cdot))=&d\Gamma(-\mathcal{A})I_n(\rho_n(t,\cdot))+D_{\lambda_d}I_{n+1}\left(\rho_{n+1}(t,\cdot)\right)\\
	&+D^{\star}_{\lambda_c}I_{n-1}\left(\rho_{n-1}(t,\cdot)\right)-\gamma I_n(\rho_n(t,\cdot)),\quad t>0, n\geq 0;\\
	I_1(\rho_1(0,\cdot))=&I_1(\zeta);\\
	I_n(\rho_n(0,\cdot))=&0,\mbox{ for all $n\neq 1$},
	\end{split}
	\end{align} 
with probability one. Here, we agree on setting $I_{-1}(\cdot)\equiv 0$.
\end{proposition}

\begin{proof}
	Using the first order contraction, see \eqref{contraction}, and symmetrized tensor product, see \eqref{symm}, we can reformulate \eqref{equation proof} as
	\begin{align}\label{equation 3}
	\begin{split}
	\partial_t\rho_0(t)=&\lambda_d\otimes_1\rho_1(t,\cdot)-\gamma\rho_0(t);\\
	\partial_t\rho_n(t,x_1,...,x_n)=&-\sum_{i=1}^n\mathcal{A}_i\rho_n(t,x_1,...,x_n)+(n+1)(\lambda_d\otimes_1\rho_{n+1}(t,\cdot))(x_1,...,x_n)\\
	&+(\lambda_c\hat{\otimes}\rho(t,\cdot))(x_1,...,x_n)-\gamma\rho_n(t,x_1,...,x_n).
	\end{split}
	\end{align}
	The continuity and symmetry of the functions $\rho_n(t,x_1,...,x_n)$, $\partial_t\rho_n(t,x_1,...,x_n)$ and $\sum_{i=1}^n\mathcal{A}_i\rho_n(t,x_1,...,x_n)$ entail their membership to $L^2_s([0,1]^n)$; this allows us to perform, for any $n\geq 1$, an $n$-th order multiple It\^o integral on both sides of the equation \eqref{equation 3} to get
	\begin{align*}
	\begin{split}
	I_n(\partial_t\rho_n(t,\cdot))=&-I_n\left(\sum_{i=1}^n\mathcal{A}_i\rho_n(t,\cdot)\right)+(n+1)I_n\left(\lambda_d\otimes_1\rho_{n+1}(t,\cdot)\right)\\
	&+I_n\left(\lambda_c\hat{\otimes}\rho_{n-1}(t,\cdot)\right)-\gamma I_n(\rho_n(t,\cdot)),
	\end{split}
	\end{align*} 
	or equivalently,
	\begin{align*}
	\begin{split}
	\partial_tI_n(\rho_n(t,\cdot))=&d\Gamma(-\mathcal{A})I_n(\rho_n(t,\cdot))+D_{\lambda_d}I_{n+1}\left(\rho_{n+1}(t,\cdot)\right)\\
	&+D^{\star}_{\lambda_c}I_{n-1}\left(\rho_{n-1}(t,\cdot)\right)-\gamma I_n(\rho_n(t,\cdot)).
	\end{split}
	\end{align*}
	The last identity holds for all $n\geq 1$, $t>0$, $\mathbb{P}$-almost surely. The initial conditions in \eqref{equation 5} are readily checked.
\end{proof}

Our next step is to construct a generalized stochastic process $\{\Phi(t)\}_{t\geq 0}$ out of the sequence $\{I_n(\rho_n(t,\cdot))\}_{n\geq 0}$ in the spirit of the Wiener-It\^o chaos expansion.

\begin{proposition}
Let $\{\rho_n\}_{n\geq 0}$ be a classical solution to \eqref{equation proof}-\eqref{initial proof}. Then, the stochastic process
\begin{align}\label{solution SDE}
	\Phi(t):=\sum_{n\geq 0}I_n(\rho_n(t,\cdot)),\quad t\geq 0,
\end{align}
belongs to $\mathtt{F}^{\star}$ and solves the differential equation
\begin{align}\label{SDE infinite}
	\begin{split}
		\partial_t\Phi(t)=&d\Gamma(-\mathcal{A})\Phi(t)+D_{\lambda_d}\Phi(t)+D^{\star}_{\lambda_c}\Phi(t)-\gamma \Phi(t), \quad t> 0,\\
		\Phi(0)=&I_1(\zeta).
	\end{split}
\end{align} 
\end{proposition}

\begin{proof}
The continuity of $\rho_n(t,\cdot)$ together with its partial derivatives up to the second order ensure that the dual pairing $\langle\langle \Phi(t),\varphi\rangle\rangle$ between $\Phi(t)$ from \eqref{solution SDE} and any element $\varphi=\sum_{n=0}^MI_n(h_n)$ of $\mathtt{F}$ is well defined thus entailing the membership of $\Phi(t)$ to $\mathtt{F}^{\star}$, for all $t\geq 0$. To prove that $\{\Phi(t)\}_{t\geq 0}$ solves \eqref{SDE infinite}, we sum equation \eqref{equation 5} over $n\geq 0$ with the convention that $I_{-1}(\cdot):=0$ (recall also that $D_{\lambda_d}I_0(\cdot)=0$): this will lead to
\begin{align*}
\partial_t\Phi(t)=&\partial_t\sum_{n\geq 0}I_n(\rho_n(t,\cdot))=\sum_{n\geq 0}\partial_tI_n(\rho_n(t,\cdot))\\
=&\sum_{n\geq 0}d\Gamma(-\mathcal{A})I_n(\rho_n(t,\cdot))+\sum_{n\geq 0}D_{\lambda_d}I_{n+1}\left(\rho_{n+1}(t,\cdot)\right)\\
&+\sum_{n\geq 0}D^{\star}_{\lambda_c}I_{n-1}\left(\rho_{n-1}(t,\cdot)\right)-\gamma \sum_{n\geq 0}I_n(\rho_n(t,\cdot))\\
=&d\Gamma(-\mathcal{A})\sum_{n\geq 0}I_n(\rho_n(t,\cdot))+D_{\lambda_d}\sum_{n\geq 0}I_{n+1}\left(\rho_{n+1}(t,\cdot)\right)\\
&+D^{\star}_{\lambda_c}\sum_{n\geq 0}I_{n-1}\left(\rho_{n-1}(t,\cdot)\right)-\gamma \sum_{n\geq 0}I_n(\rho_n(t,\cdot))\\
=&d\Gamma(-\mathcal{A})\sum_{n\geq 0}I_n(\rho_n(t,\cdot))+D_{\lambda_d}\sum_{n\geq 1}I_{n}\left(\rho_{n}(t,\cdot)\right)\\
&+D^{\star}_{\lambda_c}\sum_{n\geq 1}I_{n-1}\left(\rho_{n-1}(t,\cdot)\right)-\gamma \sum_{n\geq 0}I_n(\rho_n(t,\cdot))\\
=&d\Gamma(-\mathcal{A})\Phi(t)+D_{\lambda_d}\Phi(t)+D^{\star}_{\lambda_c}\Phi(t)-\gamma \Phi(t).
\end{align*}
Observe that the smoothness of the functions $\{\rho_n\}_{n\geq 0}$ allows for the interchange between the series and operators $\partial_t$, $d\Gamma(-\mathcal{A})$, $D_{\lambda_d}$ and $D^{\star}_{\lambda_c}$ (since we are working in the space $\mathtt{F}^{\star}$). The initial condition in \eqref{SDE infinite} is trivially verified and the proof is complete.  	
\end{proof}

\begin{remark}
The stochastic process $\{\Phi(t)\}_{t\geq 0}$ and differential equation \eqref{SDE infinite} provide a concise reformulation of the sequence $\{\rho_n\}_{n\geq 0}$ and system of partial differential equations \eqref{equation proof}-\eqref{initial proof}. In principle, one may start solving equation \eqref{SDE infinite} and then identify, via the generalized Stroock-Taylor formula \eqref{ST}, the kernels of this solution as the sequence $\{\rho_n\}_{n\geq 0}$ fulfilling \eqref{equation proof}-\eqref{initial proof}.    
\end{remark}

We now proceed with the investigation of the projected sequence $\{\Pi_N^{\otimes n}\rho_n(t,\cdot)\}_{n\geq 0}$. 

\begin{proposition}\label{proposition 2}
	Let $\{\rho_n\}_{n\geq 0}$ be a classical solution to \eqref{equation proof}-\eqref{initial proof}. Then, for any $N\geq N_0$ the sequence of random variables $\{I_n(\Pi_N^{\otimes n}\rho_n(t,\cdot))\}_{n\geq 0}$ satisfies for any $t\geq 0$ the equations
	\begin{align}\label{equation 6}
	\begin{split}
	\partial_tI_n(\Pi_N^{\otimes n}\rho_n(t,\cdot))=&d\Gamma(-\mathcal{A})I_n(\Pi_N^{\otimes n}\rho_n(t,\cdot))+D_{\lambda_d}I_{n+1}\left(\Pi_N^{\otimes (n+1)}\rho_{n+1}(t,\cdot)\right)\\
	&+D^{\star}_{\Pi_N\lambda_c}I_{n-1}\left(\Pi_N^{\otimes (n-1)}\rho_{n-1}(t,\cdot)\right)-\gamma I_n(\Pi_N^{\otimes n}\rho_n(t,\cdot));\\
	I_1(\Pi_N\rho_1(0,\cdot))=&I_1(\Pi_N\zeta);\\
	I_n(\Pi_N^{\otimes n}\rho_n(0,\cdot))=&0,\mbox{ for all $n\neq 1$},
	\end{split}
	\end{align} 
	with probability one.
\end{proposition}

\begin{proof}
We know from Proposition \ref{proposition 1} that the sequence of random variables $\{I_n(\rho_n(t,\cdot))\}_{n\geq 0}$ satisfies equation \eqref{equation 5}; moreover, according to the definition of second quantization operator we can write
\begin{align*}
I_n(\Pi_N^{\otimes n}\rho_n(t,\cdot))=\Gamma(\Pi_N)I_n(\rho_n(t,\cdot)).
\end{align*} 
Therefore,
\begin{align}\label{commutation}
\begin{split}
\partial_tI_n(\Pi_N^{\otimes n}\rho_n(t,\cdot))=&\partial_t\Gamma(\Pi_N)I_n(\rho_n(t,\cdot))\\
=&\Gamma(\Pi_N)\partial_tI_n(\rho_n(t,\cdot))\\
=&\Gamma(\Pi_N)\left[d\Gamma(-\mathcal{A})I_n(\rho_n(t,\cdot))+D_{\lambda_d}I_{n+1}\left(\rho_{n+1}(t,\cdot)\right)\right]\\
&+\Gamma(\Pi_N)\left[D^{\star}_{\lambda_c}I_{n-1}\left(\rho_{n-1}(t,\cdot)\right)-\gamma I_n(\rho_n(t,\cdot))\right]\\
=&\Gamma(\Pi_N)d\Gamma(-\mathcal{A})I_n(\rho_n(t,\cdot))+\Gamma(\Pi_N)D_{\lambda_d}I_{n+1}\left(\rho_{n+1}(t,\cdot)\right)\\
&+\Gamma(\Pi_N)D^{\star}_{\lambda_c}I_{n-1}\left(\rho_{n-1}(t,\cdot)\right)-\gamma\Gamma(\Pi_N) I_n(\rho_n(t,\cdot)).
\end{split}
\end{align}
We now have to investigate the commutation relations of $\Gamma(\Pi_N)$ with $d\Gamma(-\mathcal{A})$, $D_{\lambda_d}$ and $D^{\star}_{\lambda_c}$. Let $h\in C^2([0,1])$; then,
\begin{align}\label{1}
\begin{split}
\mathbb{E}[\Gamma(\Pi_N)d\Gamma(-\mathcal{A})I_n(\rho_n(t,\cdot))\cdot\mathcal{E}(h)]&=\mathbb{E}[d\Gamma(-\mathcal{A})I_n(\rho_n(t,\cdot))\cdot\mathcal{E}(\Pi_Nh)]\\
&=\mathbb{E}[I_n(\rho_n(t,\cdot))\cdot d\Gamma(-\mathcal{A})\mathcal{E}(\Pi_Nh)]\\
&=\mathbb{E}[I_n(\rho_n(t,\cdot))\cdot D^{\star}_{-\mathcal{A}\Pi_Nh}\mathcal{E}(\Pi_Nh)]\\
&=\mathbb{E}[I_n(\rho_n(t,\cdot))\cdot D^{\star}_{-\Pi_N\mathcal{A}h}\mathcal{E}(\Pi_Nh)].
\end{split}
\end{align}
On the other hand,
\begin{align}\label{2}
\begin{split}
\mathbb{E}[d\Gamma(-\mathcal{A})\Gamma(\Pi_N)I_n(\rho_n(t,\cdot))\cdot\mathcal{E}(h)]&=\mathbb{E}[\Gamma(\Pi_N)I_n(\rho_n(t,\cdot))\cdot d\Gamma(-\mathcal{A})\mathcal{E}(h)]\\
&=\mathbb{E}[\Gamma(\Pi_N)I_n(\rho_n(t,\cdot))\cdot D^{\star}_{-\mathcal{A}h}\mathcal{E}(h)]\\
&=\mathbb{E}[I_n(\rho_n(t,\cdot))\cdot \Gamma(\Pi_N)D^{\star}_{-\mathcal{A}h}\mathcal{E}(h)]\\
&=\mathbb{E}[I_n(\rho_n(t,\cdot))\cdot D^{\star}_{-\Pi_N\mathcal{A}h}\mathcal{E}(\Pi_Nh)].
\end{split}
\end{align}
Comparing the first and last members of \eqref{1} and \eqref{2} we deduce that
\begin{align*}
\mathbb{E}[\Gamma(\Pi_N)d\Gamma(-\mathcal{A})I_n(\rho_n(t,\cdot))\cdot\mathcal{E}(h)]=\mathbb{E}[d\Gamma(-\mathcal{A})\Gamma(\Pi_N)I_n(\rho_n(t,\cdot))\cdot\mathcal{E}(h)],
\end{align*}
for all $h\in C^2([0,1])$ and hence that
\begin{align}\label{a1}
\Gamma(\Pi_N)d\Gamma(-\mathcal{A})I_n(\rho_n(t,\cdot))=d\Gamma(-\mathcal{A})\Gamma(\Pi_N)I_n(\rho_n(t,\cdot)),\quad \mathbb{P}-\mbox{a.s}.
\end{align}
We now consider the commutation between $\Gamma(\Pi_N)$ and $D_{\lambda_d}$:
\begin{align*}
\mathbb{E}[\Gamma(\Pi_N)D_{\lambda_d}I_n(\rho_n(t,\cdot))\cdot\mathcal{E}(h)]&=\mathbb{E}[D_{\lambda_d}I_n(\rho_n(t,\cdot))\cdot\mathcal{E}(\Pi_Nh)]\\
&=\mathbb{E}[nI_{n-1}(\lambda_d\otimes_1\rho_n(t,\cdot))\cdot\mathcal{E}(\Pi_Nh)]\\
&=n\langle \lambda_d\otimes_1\rho_n(t,\cdot),(\Pi_Nh)^{\otimes {n-1}}\rangle_{L^2([0,1]^{n-1})}\\
&=n\langle \lambda_d\otimes_1\Pi_N^{\otimes (n-1)}\rho_n(t,\cdot),h^{\otimes {n-1}}\rangle_{L^2([0,1]^{n-1})},
\end{align*}
while
\begin{align*}
\mathbb{E}[D_{\lambda_d}\Gamma(\Pi_N)I_n(\rho_n(t,\cdot))\cdot\mathcal{E}(h)]&=\mathbb{E}[D_{\lambda_d}I_n(\Pi_N^{\otimes n}\rho_n(t,\cdot))\cdot\mathcal{E}(h)]\\
&=\mathbb{E}[nI_{n-1}(\lambda_d\otimes_1\Pi_N^{\otimes n}\rho_n(t,\cdot))\cdot\mathcal{E}(h)]\\
&=\mathbb{E}[nI_{n-1}(\Pi_N\lambda_d\otimes_1\Pi_N^{\otimes (n-1)}\rho_n(t,\cdot))\cdot\mathcal{E}(h)]\\
&=n\langle \Pi_N\lambda_d\otimes_1\Pi_N^{\otimes (n-1)}\rho_n(t,\cdot),h^{\otimes {n-1}}\rangle_{L^2([0,1]^{n-1})}.
\end{align*}
Therefore,
\begin{align}\label{a2}
\Gamma(\Pi_N)D_{\lambda_d}I_n(\rho_n(t,\cdot))=D_{\lambda_d}\Gamma(\Pi_N)I_n(\rho_n(t,\cdot)),
\end{align}
if $\lambda_d=\Pi_N\lambda_d$; but this is the case since we are assuming $N\geq N_0$ (recall Assumption \ref{assumption 2}). Lastly, we investigate the commutation between $\Gamma(\Pi_N)$ and $D^{\star}_{\lambda_c}$:
\begin{align*}
\mathbb{E}[\Gamma(\Pi_N)D^{\star}_{\lambda_c}I_{n-1}\left(\rho_{n-1}(t,\cdot)\right)\cdot\mathcal{E}(h)]=&\mathbb{E}[D^{\star}_{\lambda_c}I_{n-1}\left(\rho_{n-1}(t,\cdot)\right)\cdot\mathcal{E}(\Pi_N h)]\\
=&\mathbb{E}[I_{n-1}\left(\rho_{n-1}(t,\cdot)\right)\cdot D_{\lambda_c}\mathcal{E}(\Pi_N h)]\\
=&\mathbb{E}[I_{n-1}\left(\rho_{n-1}(t,\cdot)\right)\cdot \langle\lambda_c,\Pi_Nh\rangle\mathcal{E}(\Pi_N h)]\\
=&\mathbb{E}[I_{n-1}\left(\rho_{n-1}(t,\cdot)\right)\cdot \langle\Pi_N\lambda_c,h\rangle\mathcal{E}(\Pi_N h)]\\
=&\mathbb{E}[\Gamma(\Pi_N)I_{n-1}\left(\rho_{n-1}(t,\cdot)\right)\cdot \langle\Pi_N\lambda_c,h\rangle\mathcal{E}(h)]\\
=&\mathbb{E}[D^{\star}_{\Pi_N\lambda_c}\Gamma(\Pi_N)I_{n-1}\left(\rho_{n-1}(t,\cdot)\right)\cdot\mathcal{E}(h)];
\end{align*}
comparing the first and last members we can conclude that
\begin{align}\label{a3}
\Gamma(\Pi_N)D^{\star}_{\lambda_c}I_{n-1}\left(\rho_{n-1}(t,\cdot)\right)=D^{\star}_{\Pi_N\lambda_c}\Gamma(\Pi_N)I_{n-1}\left(\rho_{n-1}(t,\cdot)\right),
\end{align}
almost surely. Now, using \eqref{a1}-\eqref{a2}-\eqref{a3} in \eqref{commutation} we obtain
\begin{align*}
\partial_tI_n(\Pi_N^{\otimes n}\rho_n(t,\cdot))=&\Gamma(\Pi_N)d\Gamma(-\mathcal{A})I_n(\rho_n(t,\cdot))+\Gamma(\Pi_N)D_{\lambda_d}I_{n+1}\left(\rho_{n+1}(t,\cdot)\right)\\
&+\Gamma(\Pi_N)D^{\star}_{\lambda_c}I_{n-1}\left(\rho_{n-1}(t,\cdot)\right)-\gamma\Gamma(\Pi_N) I_n(\rho_n(t,\cdot))\\
=&d\Gamma(-\mathcal{A})I_n(\Pi_N^{\otimes n}\rho_n(t,\cdot))+D_{\lambda_d}I_{n+1}\left(\Pi_N^{\otimes (n+1)}\rho_{n+1}(t,\cdot)\right)\\
&+D^{\star}_{\Pi_N\lambda_c}I_{n-1}\left(\Pi_N^{\otimes (n-1)}\rho_{n-1}(t,\cdot)\right)-\gamma I_n(\Pi_N^{\otimes n}\rho_n(t,\cdot)),
\end{align*}
which corresponds to \eqref{equation 6}.
\end{proof}

The last step of our construction consists in rewriting equations \eqref{equation 6} from a standard partial differential equation's perspective. 

\begin{proposition}
	 For any $n\geq 1$, 
	\begin{align}\label{representation}
		I_n(\Pi_N^{\otimes n}\rho_n(t,\cdot))=\varphi_{n,N}\left(t,I_1(\xi_1),...,I_1(\xi_N)\right),
	\end{align}
	where $\varphi_{n,N}:[0,+\infty[\times\mathbb{R}^N\to\mathbb{R}$ is a polynomial of degree $n$ in the variables $I_1(\xi_1)$,..., $I_1(\xi_N)$. With this representation, equations \eqref{equation 6} read
\begin{align}\label{PDE for representation}
	\begin{split}
	(\partial_t\varphi_{n,N})\left(t,z\right)=&\sum_{k=1}^N\alpha_k\partial^2_{z_k}\varphi_{n,N}\left(t,z\right)-\sum_{k=1}^N\alpha_kz_k\partial_{z_k}\varphi_{n,N}\left(t,z\right)+\sum_{k=1}^Nd_k\partial_{z_k}\varphi_{n+1,N}\left(t,z\right)\\
	&+\varphi_{n-1,N}\left(t,z\right)\cdot \left(\sum_{k=1}^Nc_kz_k\right)-\sum_{k=1}^Nc_k\partial_{z_k}\varphi_{n-1,N}\left(t,z\right)-\gamma\varphi_{n,N}(t,z),
	\end{split}	
\end{align}
for $n\geq 1$, $t\geq 0$ and $z=(z_1,...,z_N)\in\mathbb{R}^{N}$. Here we set $\varphi_{-1,N}\equiv 0$.
\end{proposition}

\begin{proof}
We start observing (see for instance \cite{Janson}) that
\begin{align*}
	I_n(\Pi_N^{\otimes n}\rho_n(t,\cdot))=\Gamma(\Pi_N)	I_n(\rho_n(t,\cdot))=\mathbb{E}[I_n(\rho_n(t,\cdot))|\sigma(I_1(\xi_1),...,I_1(\xi_N))],
\end{align*}
where the right hand side above stands for the conditional expectation of the random variable $I_n(\rho_n(t,\cdot))$ with respect to the sigma-algebra generated by the random variables $\{I_1(\xi_1),...,I_1(\xi_N)\}$. It is also well known from the theory of multiple It\^o integrals that $I_n(\rho_n(t,\cdot))$ can be written as an infinite linear combination of polynomials of degree $n$ in the variables $I_1(\xi_1), I_1(\xi_2),...$; the action of the conditional expectation above reduce that linear combination to a finite number of terms in the variables $\{I_1(\xi_1),...,I_1(\xi_N)\}$. This proves identity \eqref{representation}. Let us now derive equation \eqref{PDE for representation}; according to formula \eqref{dsq} we have
\begin{align*}
d\Gamma(-\mathcal{A})I_n(\Pi_N^{\otimes n}\rho_n(t,\cdot))=&d\Gamma(-\mathcal{A})\varphi_{n,N}\left(t,I_1(\xi_1),...,I_1(\xi_N)\right)\\
=&-\delta\left(\mathcal{A}D_{\cdot}\varphi_{n,N}\left(t,I_1(\xi_1),...,I_1(\xi_N)\right)\right)\\
=&-\delta\left(\sum_{k=1}^N(\partial_{z_k}\varphi_{n,N})\left(t,I_1(\xi_1),...,I_1(\xi_N)\right)\mathcal{A}\xi_k(\cdot)\right)\\
=&-\sum_{k=1}^N\alpha_k\delta\left((\partial_{z_k}\varphi_{n,N})\left(t,I_1(\xi_1),...,I_1(\xi_N)\right)\xi_k(\cdot)\right)\\
=&-\sum_{k=1}^N\alpha_k\delta(\xi_k)\cdot(\partial_{z_k}\varphi_{n,N})\left(t,I_1(\xi_1),...,I_1(\xi_N)\right)\\
&+\sum_{k=1}^N\alpha_k\partial^2_{z_k}\varphi_{n,N}\left(t,I_1(\xi_1),...,I_1(\xi_N)\right);
\end{align*}  
moreover, using the chain rule for Malliavin derivatives we get
\begin{align*}
D_{\lambda_d}I_{n+1}\left(\Pi_N^{\otimes (n+1)}\rho_{n+1}(t,\cdot)\right)=&D_{\lambda_d}\varphi_{n+1,N}\left(t,I_1(\xi_1),...,I_1(\xi_N)\right)\\
=&\sum_{k=1}^N(\partial_{z_k}\varphi_{n+1,N})\left(t,I_1(\xi_1),...,I_1(\xi_N)\right)\langle\lambda_d,\xi_k\rangle_{L^2([0,1])}\\
=&\sum_{k=1}^Nd_k(\partial_{z_k}\varphi_{n+1,N})\left(t,I_1(\xi_1),...,I_1(\xi_N)\right),
\end{align*}
and
\begin{align*}
D^{\star}_{\Pi_N\lambda_c}I_{n-1}\left(\Pi_N^{\otimes (n-1)}\rho_{n-1}(t,\cdot)\right)=&D^{\star}_{\Pi_N\lambda_c}\varphi_{n-1,N}\left(t,I_1(\xi_1),...,I_1(\xi_N)\right)\\
=&\varphi_{n-1,N}\left(t,I_1(\xi_1),...,I_1(\xi_N)\right)\cdot I_1(\Pi_N\lambda_c)\\
&-D_{\Pi_N\lambda_c}\varphi_{n-1,N}\left(t,I_1(\xi_1),...,I_1(\xi_N)\right)\\
=&\varphi_{n-1,N}\left(t,I_1(\xi_1),...,I_1(\xi_N)\right)\cdot \left(\sum_{k=1}^Nc_kI_1(\xi_k)\right)\\
&-\sum_{k=1}^N(\partial_{z_k}\varphi_{n-1,N})\left(t,I_1(\xi_1),...,I_1(\xi_N)\right)c_k.
\end{align*}
Note that in the second equality above we made use of identity \eqref{gaussian commutation}. Combining all the previous identities we can rewrite equations \eqref{equation 6} as
\begin{align*}
(\partial_t\varphi_{n,N})\left(t,I_1(\xi_1),...,I_1(\xi_N)\right)=&\sum_{k=1}^N\alpha_k\partial^2_{z_k}\varphi_{n,N}\left(t,I_1(\xi_1),...,I_1(\xi_N)\right)\\
&-\sum_{k=1}^N\alpha_k\delta(\xi_k)\cdot(\partial_{z_k}\varphi_{n,N})\left(t,I_1(\xi_1),...,I_1(\xi_N)\right)\\
&+\sum_{k=1}^Nd_k(\partial_{z_k}\varphi_{n+1,N})\left(t,I_1(\xi_1),...,I_1(\xi_N)\right)\\
&+\varphi_{n-1,N}\left(t,I_1(\xi_1),...,I_1(\xi_N)\right)\cdot \left(\sum_{k=1}^Nc_kI_1(\xi_k)\right)\\
&-\sum_{k=1}^N(\partial_{z_k}\varphi_{n-1,N})\left(t,I_1(\xi_1),...,I_1(\xi_N)\right)c_k\\
&-\gamma\varphi_{n,N}\left(t,I_1(\xi_1),...,I_1(\xi_N)\right).
\end{align*}
This corresponds to \eqref{PDE for representation} upon replacing $I_1(\xi_k)$ with $z_k$, for $k=1,...,N$.
\end{proof}

We are now ready to prove Theorem \ref{main theorem}. Let 
\begin{align}\label{final}
u(t,I_1(\xi_1),...,I_1(\xi_N)):=\sum_{n\geq 0}\varphi_{n,N}\left(t,I_1(\xi_1),...,I_1(\xi_N)\right).
\end{align}
According to the last proposition, summing over $n\geq 0$ equations \eqref{PDE for representation} we see that the function $u$ in \eqref{final} solves
\begin{align*}
\begin{split}
\partial_tu(t,z)&=\sum_{k=1}^N\alpha_k\partial^2_{z_k}u(t,z)+\sum_{k=1}^N\left(d_k-c_k-\alpha_k z_k\right)\partial_{z_k}u(t,z)+\left(\sum_{k=1}^Nc_kz_k-\gamma\right)u(t,z)\\
u(0,z)&=\sum_{k=1}^N\zeta_kz_k.
\end{split}
\end{align*}
On the other hand, by construction 
\begin{align*}
\sum_{n\geq 0}\varphi_{n,N}\left(t,I_1(\xi_1),...,I_1(\xi_N)\right)=\sum_{n\geq 0}I_n\left(\Pi_N^{\otimes n}\rho_n(t,\cdot)\right);
\end{align*}
this gives
\begin{align*}
u(t,I_1(\xi_1),...,I_1(\xi_N))=\sum_{n\geq 0}I_n\left(\Pi_N^{\otimes n}\rho_n(t,\cdot)\right).
\end{align*}
Thus, the kernels in the Wiener-It\^o chaos expansion of $u(t,I_1(\xi_1),...,I_1(\xi_N))$, which in general can be represented via the Stroock-Taylor formula \eqref{ST}, coincide with the sequence $\{\Pi_N^{\otimes n}\rho_n(t,\cdot)\}_{n\geq 0}$. Therefore,
\begin{align*}
\Pi_N^{\otimes n}\rho_n(t,x_1,...,x_n)=&\frac{1}{n!}\mathbb{E}\left[D_{x_1,...,x_n}u(t,I_1(\xi_1),...,I_1(\xi_N))\right]\\
=&\frac{1}{n!}\sum_{j_1,...j_n=1}^N\mathbb{E}\left[\left(\partial_{z_{j_1}}\cdot\cdot\cdot \partial_{z_{j_n}}u\right)(t,Z)\right]\xi_{j_1}(x_1)\cdot\cdot\cdot\xi_{j_n}(x_n),
\end{align*}
where 
\begin{align*}
\mathbb{E}\left[\left(\partial_{z_{j_1}}\cdot\cdot\cdot \partial_{z_{j_n}}u\right)(t,Z)\right]=\int_{\mathbb{R}^N}\left(\partial_{z_{j_1}}\cdot\cdot\cdot \partial_{z_{j_n}}u\right)(t,z)(2\pi)^{-N/2}e^{-\frac{|z|^2}{2}}dz.
\end{align*}
This completes the proof of our main result (recall that the orthonormality of the functions $\xi_1,...,\xi_N$ implies that $I_1(\xi_1),...,I_1(\xi_N)$ are independent standard Gaussian random variables).

\begin{remark}
Following the previous construction, we could formally associate equation \eqref{SDE infinite} with
\begin{align}\label{last}
	\begin{split}
		\partial_tu(t,z)&=\sum_{k\geq 1}\alpha_k\partial^2_{z_k}u(t,z)+\sum_{k\geq 1}\left(d_k-c_k-\alpha_k z_k\right)\partial_{z_k}u(t,z)+\left(\sum_{k\geq 1}c_kz_k-\gamma\right)u(t,z)\\
		u(0,z)&=\zeta(z),
	\end{split}
\end{align}
where now $u(t,z)=u(t,z_1,z_2,....)$ is a function of infinitely many variables, and obtain the identity
\begin{align*}
\Phi(t)=u\left(t,I_1(\xi_1), I_1(\xi_2),...\right).
\end{align*} 
From this point of view, the use of the projection operator $\Pi_N$ is needed for reducing \eqref{last} to a standard partial differential equation with a finite number of spatial variables.  
\end{remark}


\begin{thebibliography}{9}

\bibitem{Allen}
E. Allen, \emph{Modelling with Itô Stochastic Differential Equations}, Springer-Verlag, London, 2007.

\bibitem{Bogachev}
V. I. Bogachev, \emph{Gaussian Measures}, American Mathematical Society, Providence, 1998.
	
\bibitem{del Razo}
M. J. del Razo, D. Fr\"omberg, A. V. Straube, C. Sch\"utte, F. H\"ofling, and S. Winkelmann.
A probabilistic framework for particle-based reaction-diffusion dynamics using classical
Fock space representations, arXiv:2109.13616, 2021.	

\bibitem{Gardiner}
C. Gardiner, \emph{Stochastic Methods: A Handbook for the Natural and Social Sciences - IV edition}, Springer Series in Synergetics, Springer, Berlin Heidelberg, 2009.

\bibitem{Gillespie}
D. T. Gillespie, A rigorous derivation of the chemical master equation, \emph{Physica A: Statistical Mechanics and its Applications}, \textbf{188} (1992) pp. 404-425.
		
\bibitem{Janson}
S. Janson, \emph{Gaussian Hilbert spaces}, Cambridge Tracts in Mathematics, 129. Cambridge University Press, Cambridge, 1997.

\bibitem{KS}
I. Karatzas and S. E. Shreve, \emph{Brownian motion and stochastic calculus}, Springer-Verlag, New York, 1991.

\bibitem{Lecca}
P. Lecca, I. Laurenzi, I. and F. Jordan, \emph{Deterministic Versus Stochastic Modelling in Biochemistry and Systems Biology}, Oxford: Woodhead Publishing, 2013.

\bibitem{MR}
R. Mikulevicius and B. L. Rozovskii, On unbiased stochastic Navier-Stokes equations,
\emph{Probability Theory and Related Fields} \textbf{154} (2012) pp. 787-834.

\bibitem{Nualart}
D. Nualart, \emph{Malliavin calculus and Related Topics - II Edition}, Springer, New York,
2006.

\bibitem{Van Kampen}
N. G. Van Kampen, \emph{Stochastic Processes in Physics and Chemistry}, Elsevier,
Amsterdam, 1992.

\end{thebibliography}
\end{document}